\UseRawInputEncoding
\documentclass[final,11pt,a4paper,reqno]{amsart}
\usepackage{stmaryrd}
\usepackage[left]{showlabels}
\usepackage[english]{babel}
\usepackage{pst-grad} 
\usepackage{pst-plot} 
\usepackage{pstricks}
\usepackage{amsmath,amssymb,mathrsfs,amsthm, mathtools}
\usepackage{tikz} 
\usepackage{mathcomp,wasysym}  
\usepackage{graphicx,subfigure}  
\usepackage[all]{xy} \xyoption{arc} \xyoption{color}
\usepackage{epsfig}
\usepackage{cite}
\usepackage[a4paper,
left=2.5cm, right=2.5cm,
top=3cm, bottom=3cm]{geometry}
\usepackage[bookmarks]{hyperref}
\hypersetup{
    colorlinks=true,
    linkcolor=cyan,
    filecolor=magenta,      
    urlcolor=cyan,
    citecolor=blue,
    pdftitle={},
    pdfpagemode=FullScreen,
    }
\usepackage{pgfplots}

\newcommand{\norm}[1]{\left\| #1 \right\|}
\newcommand{\av}[1]{\left| #1 \right|}

\newcommand{\RR}{\mathbb R}

\newcommand{\pa}{\partial}

\normalsize
\normalsize
\newtheorem{satz}{Proposition}[section]
\newtheorem{lem}[satz]{Lemma} 
\newtheorem{theorem}{Theorem}

\newtheorem{prop}[satz]{Proposition} 
\newtheorem{lemma}[satz]{Lemma} 

\theoremstyle{definition}
\newtheorem{rem}{Remark}

\pgfplotsset{
    every axis/.append style={
            axis x line=middle,
            axis y line=middle,
            xlabel={$x$},
            ylabel={$y$},
            axis line style={->},
        },
    marya/.style={color=black,thick,mark=none},
    soldot/.style={color=black,only marks,mark=*},
    holdot/.style={color=black,fill=white,only marks,mark=*},
    grid style={dotted,gray},
}

\tikzset{>=stealth}

\title[Asymptotic models for cold plasma equation]{Derivation and well-posedness for asymptotic models of cold plasmas}

\author[D. Alonso-Or\'{a}n]{Diego Alonso-Or\'{a}n}
\email{dalonsoo@ull.edu.es}
\address{Departamento de An\'{a}lisis Matem\'{a}tico y Instituto de Matem\'{a}ticas y Aplicaciones (IMAULL), Universidad de La Laguna C/. Astrof\'{i}sico Francisco S\'{a}nchez s/n, 38200 - La Laguna, Spain.}
\author[A. Dur\'{a}n]{Angel Dur\'{a}n}
\email{angeldm@uva.es}
\address{Applied Mathematics Department, University of Valladolid, 47011 Valladolid, Spain}
\author[R. Granero-Belinch\'{o}n]{Rafael Granero-Belinch\'{o}n}
\email{rafael.granero@unican.es}
\address{Departamento  de  Matem\'aticas,  Estad\'istica  y  Computaci\'on,  Universidad  de Cantabria.  Avda.  Los  Castros  s/n,  Santander,  Spain.}

\begin{document}
\begin{abstract}
In this paper we derive three new asymptotic models for a hyperbolic-hyperbolic-elliptic system of PDEs describing the motion of a collision-free plasma in a magnetic field.  The first of these models takes the form of a non-linear and non-local Boussinesq system (for the ionic density and velocity) while the second is a non-local wave equation (for the ionic density). Moreover, we derive a unidirectional asymptotic model of the latter which is closely related to the well-known Fornberg-Whitham equation. We also provide the well-posedness of these asymptotic models in Sobolev spaces. To conclude, we demonstrate the existence of a class of initial data which exhibit wave breaking for the unidirectional model. 
\end{abstract}



\maketitle
{\small
\tableofcontents}

\allowdisplaybreaks

\section{Introduction}
The motion of a cold plasma in a magnetic field consisting of singly-charged particles can be described by the following system of PDEs \cite{berezin1964theory, gardner1960similarity}
\begin{subequations}\label{eq:1}
\begin{align}
n_t+(un)_x&=0,\\
u_t+uu_x+\frac{bb_x}{n}&=0,\\
b-n-\left(\frac{b_x}{n}\right)_x&=0
\end{align}
\end{subequations}
where $n,u$ and $b$ are the ionic density, the ionic velocity and the magnetic field, respectively.  Moreover, it has also been used as a simplified model to describe the motion of collission-free two fluid model where the electron inertial, charge separation and displacement current are neglected and the Poisson equation (\ref{eq:1}c) is initially satisfied, \cite{berezin1964theory, kakutani1968reductive}. In (\ref{eq:1}) the spatial domain $\Omega$ is either $\Omega=\mathbb{R}$ or $\Omega=\mathbb{S}^1$ (\emph{i.e.} $x\in\mathbb{R}$ or $x\in [-\pi,\pi]$ with periodic boundary conditions) and the time variable satisfies $t\in [0,T]$ for certain $0<T\leq\infty$. The corresponding initial-value problem (ivp) consists of the system \eqref{eq:1} along with initial conditions
\begin{equation}\label{initialdata}
n(x,0)=n_0(x),\;u(x,0)=u_0(x),
\end{equation}
which are assumed to be smooth enough for the purposes of the work.

System \eqref{eq:1} was introduced by Gardner \& Morikawa \cite{gardner1960similarity}. Furthermore, Gardner \& Morikawa formally showed that the solutions of \eqref{eq:1} converge to solutions of the Korteveg-de Vries equation (see also the paper by Su \& Gardner \cite{su1969korteweg}). Berezin \& Karpman extended this formal limit to the case where the wave propagates at certain angles with respect to the magnetic field \cite{berezin1964theory}, \emph{i.e.} for angles satisfying certain size conditions. Later on, Kakutani, Ono, Taniuti \& Wei \cite{kakutani1968reductive} removed the hypothesis on the angle. This formal KdV limit was recently justified by Pu \& Li \cite{pu2019kdv}.

\subsection{Contributions and main results}
The purpose of this paper is two-fold. First, we derive three asymptotic models for the hyperbolic-hyperbolic-elliptic system of PDEs describing the motion of a collision-free plasma in a magnetic field given in \eqref{eq:1}. The method to obtain the new asymptotic models relies on a multi-scale expansion (cf.  e.~g. \cite{AloranWaves, CL14,CL15,CLS12,CGSW18,granero2018asymptotic,GS}) which reduces the full system \eqref{eq:1} to a cascade of linear equations which can be closed up to some order of precision. 

More specifically, writing
\begin{equation}\label{adg1}
n=1+N,\;U=u,\; b=1+B,
\end{equation}
and, for $\varepsilon>0$, introducing the formal expansions
\begin{equation}\label{adg2}
N=\sum_{\ell=0}^\infty \varepsilon^{\ell+1} N^{(\ell)},\;B=\sum_{\ell=0}^\infty \varepsilon^{\ell+1} B^{(\ell)},\;U=\sum_{\ell=0}^\infty \varepsilon^{\ell+1} U^{(\ell)},
\end{equation}
then the first model is an $O(\epsilon^{2})$ approximation of \eqref{eq:1} and takes the form of 
the following Boussinesq type system 
\begin{subequations}\label{eq:Boussinesq}
\begin{align}
h_t+(hv)_x+v_x&=0,\\
v_t+vv_x+\left[\mathscr{L},\mathscr{N}h\right]h+\mathscr{N}h&=0,
\end{align}
\end{subequations}
for $h=\varepsilon N^{(0)}+\varepsilon^2 N^{(1)},\;v=\varepsilon U^{(0)}+\varepsilon^2 U^{(1)}$. The nonlocal terms in (\ref{eq:Boussinesq}) are given by
\begin{equation}\label{adg3}
\mathscr{L}=-\partial_x^2(1-\partial_x^2)^{-1},\;
\mathscr{N}=\partial_x(1-\partial_x^2)^{-1}\;(\text{so }\partial_x\mathscr{N}=-\mathscr{L}),
\end{equation}
which are Fourier multiplier operators with symbols
\begin{equation}\label{adg4}
\widehat{\mathscr{L}h}(\xi)=\frac{\xi^2}{1+\xi^2}\hat{h}(\xi),
\; \widehat{\mathscr{N}h}(\xi)=\frac{i\xi}{1+\xi^2}\hat{h}(\xi).
\end{equation}
(where $\widehat{h}(\xi)$ denotes the Fourier transform of $h$ at $\xi$) and $\left[\mathscr{L},\cdot\right]\cdot$ denotes the commutator
\begin{equation}\label{adg4b}
\left[\mathscr{L},f\right]g=\mathscr{L}(fg)-f\mathscr{L}g.
\end{equation}
The extra assumption $U^{(0)}=N^{(0)}$ in (\ref{adg2}) leads to the formal derivation of the second asymptotic model, as a bidirectional single non-local wave equation
\begin{align}\label{wave}
h_{tt}+\mathscr{L}h&=\left(hh_x+\left[\mathscr{L},\mathscr{N}h\right]h\right)_x-2\left(hh_t\right)_x.
\end{align}

The formal reduction of (\ref{wave}) to the corresponding unidirectional version, cf. \cite{Whitham74}, yields
\begin{equation}\label{wave:uni}
h_{t}=-\frac{1}{2}\left(3 hh_{x}-[\mathscr{L},\mathscr{N}h]h - \mathscr{N}h-h_{x} \right),
\end{equation}
being the third asymptotic model introduced in the present paper. We note that the
unidirectional equation \eqref{wave:uni} has strong similarities with the well-known equation
\begin{equation}\label{FW}
u_{t}+\frac{3}{2} uu_{x}= \mathscr{N}u,
\end{equation}
proposed by Fornberg \& Whitham as a model for breaking waves \cite{Fornberg-Whitham-78}. The latter equation has been intensively studied during the last decades and several results regarding the well-posedness of the ivp in different functional spaces as well as various wave-breaking criteria have appeared in the literature \cite{Itasaka-2021, Yang-2021, Haziot-2017, Hormann-2021,HolmesThompson-2017,Holmes-2016}. A significant difference between the structure of equation \eqref{FW} and the unidirectional equation \eqref{wave:uni} derived in the present paper is the emergence of the nonlocal commutator-type term. 

The second purpose of this work is the study of several analytical properties of the models (\ref{eq:Boussinesq}), (\ref{wave}), and (\ref{wave:uni}). They are mainly concerned with the existence of conserved quantities, well-posedness (in the sense of existence and uniqueness of solutions of the corresponding ivp), and the formation of special solutions. This paper will focus on the first two points, while the existence and dynamics of solitary-wave solutions will be the object of a separate forthcoming study.

Specifically, the results shown in this paper can be summarized as follows:
\begin{itemize}
\item The system (\ref{eq:Boussinesq}) and the unidirectional model (\ref{wave:uni}) admit several quantities preserved by the solutions in suitable spaces, including  a Hamiltonian structure. On the other hand, (\ref{wave}) can be written in a conservation form, leading in a natural way to the existence of a conserved quantity.
\item The system (\ref{eq:Boussinesq}) is locally well posed on a modified Sobolev space involving the operator $\mathscr{L}$.
\item The bidirectional non-local wave equation \eqref{wave} has a unique local solution close to the equilibrium and for initial data with sufficiently small $L^{\infty}$ norm.
\item The ivp for the equation \eqref{wave:uni} is locally well-posed in Sobolev spaces $H^{s}(\RR)$ for $s>\frac{3}{2}$. Furthermore, a blow-up criterion for the solution by means of a logarithmic Sobolev inequality is provided. In addition, smooth solutions of \eqref{wave:uni} are shown to exhibit  wave breaking under a suitable hypothesis on the initial condition.
\end{itemize}

\subsection{Structure of the paper} 
In Section \ref{sec:deri} we present the asymptotic derivation of the three models studied in this paper from the motion of a cold plasma by means of a multi-scale expansion. Section \ref{sec:quantities} is devoted to the study of conservation properties of the models derived in Section \ref{sec:deri}. Focused on well-posedness, the nonlocal Boussinesq system (\ref{eq:Boussinesq}) is analyzed in Section \ref{sec:nonlocalBouss}, while local existence for the bidirectional non-local wave equation \eqref{wave} close to the equilibrium state is studied in Section \ref{sec:bidirec}. Concerning the unidirectional model \eqref{wave:uni}, well-posedness and blow-up criteria for the solutions are derived in Section \ref{sec:uni:wp}. These results are finished off in Section \ref{sec:breaking}, where wave breaking of some smooth solutions for the unidirectional model is shown.

\subsection{Preliminaries and notation}
Let us next introduce some notation that will be used throughout the rest of the paper.  

\subsubsection*{\textsf{The functional spaces}}
For $1\leq p\leq\infty$, let $L^{p}=L^{p}(\RR)$ be the usual normed space of $L^{p}$-functions on $\RR$ with $||\cdot ||_{p}$ as the associated norm. For $s\in\RR$, the inhomogeneous Sobolev space $H^{s}=H^s(\RR)$ is defined as
	\begin{align*}
		H^s(\RR)\triangleq\left\{f\in L^2(\RR):\|f\|_{H^s(\RR)}^2=\int_\RR(1+\xi^2)^s|\widehat{f}(\xi)|^2<+\infty\right\},
	\end{align*}
with norm 
\[ \norm{f}_{H^s}^{2}=\norm{f}_{L^2}^2+\norm{f}_{\dot{H}^s}^{2}, \]
where $\norm{f}_{\dot{H}^s}=\norm{\Lambda^{s}f}_{L^2}$ and $\Lambda^{s}$ is defined by
$\widehat{\Lambda^sf}(\xi)=|\xi|^{s}\widehat{f}(\xi)$, where $\widehat{f}$ is the Fourier transform of $f$.

The space of functions with bounded mean oscillation $\text{BMO}(\RR)$ (cf. \cite{Stein1, Stein2}) is defined by
\begin{align*}
		\text{BMO}(\RR)\triangleq\left\{f\in L^{1}_{\text{loc}}(\RR):\|f\|_{\text{BMO}(\RR)}=\displaystyle\sup_{r>0,x_{0}\in\RR}\int_{x_{0}-r}^{x_{0}+r} \av{f(x)-\bar{f}(x)}\  dx<+\infty\right\},
	\end{align*}
where $\bar{f}(x)=\frac {1}{2r}\int_{x_{0}-r}^{x_{0}+r} f(y) \ dy$.\\	

Next, let us introduce two lemmas with useful estimates regarding Sobolev spaces. The first one deals with the so called Kato-Ponce commutator estimate.
\begin{lem}[\cite{Kato-Ponce-1988-CPAM, Kenig-Ponce-Vega-1991-JAMS}]\label{Kato-Ponce commutator estimate}
		If $f,g\in H^s\bigcap W^{1,\infty}$ with $s>0$, then for $p,p_i\in(1,\infty)$ with $i=1,\ldots,4$ and 
		$\frac{1}{p}=\frac{1}{p_1}+\frac{1}{p_2}=\frac{1}{p_3}+\frac{1}{p_4}$, we have
		$$
		\|\left[\Lambda^s,f\right]g\|_{L^p}\leq C_{s,p}(\|\pa_{x} f\|_{L^{p_1}}\|\Lambda^{s-1}g\|_{L^{p_2}}+\|\Lambda^sf\|_{L^{p_3}}\|g\|_{L^{p_4}}),$$
		and
		$$\|\Lambda^s(fg)\|_{L^p}\leq C_{s,p}(\|f\|_{L^{p_1}}\|\Lambda^s g\|_{L^{p_2}}+\|\Lambda^s f\|_{L^{p_3}}\|g\|_{L^{p_4}}).$$
	\end{lem}
The second gives a logarithmic Sobolev inequality.
\begin{lem}[\cite{Dong-2008-JFA}]\label{Sob:log}
Let $s>\frac{1}{2}$. There exists a constant $C=C(s)>0$ such that
\begin{equation*}
\norm{f}_{L^{\infty}}\leq C\left(1+\norm{f}_{\text{BMO}}\big[1+\displaystyle\log(1+\norm{f}_{H^{s}(\RR)}) \big] \right),
\end{equation*}
holds for all $f\in H^{s}.$
\end{lem}
	
\subsubsection*{\textsf{The Helmholtz operator}}

We denote by $\mathscr{Q}$ the operator $(1-\partial_{x}^{2})^{-1}$ which acting on functions $f\in L^{2}(\RR)$ has the representation 
\begin{equation} \label{repre:Q}
\mathscr{Q}f(x)=[G\star f](x)=\int_{\RR} G(x-\eta) f(\eta) \  d\eta ,  \quad G(x)=\frac{1}{2}e^{-|x|}, \quad x\in\RR.
\end{equation}
Furthermore, the Fourier symbol of $\mathscr{Q}$ is
$$\widehat{\mathscr{Q}f}(\xi)=\frac{1}{1+|\xi|^{2}}\widehat{f}(\xi),$$
and by a simple computation we have that $\mathscr{Q}f\in H^{2}$ if $f\in L^{2}$ and
\begin{equation}\label{comp:helmholtz}
\pa_{x}^{2}\mathscr{Q}f(x)=\left( \mathscr{Q}-\text{I}\right)f(x),\; x\in\mathbb{R},
\end{equation}
where $\text{I}$ denotes the identity operator.

\subsubsection*{\textsf{Constants}}

Throughout the paper $C = C(\cdot)$ will denote a positive constant that may depend on fixed parameters  and  $x \lesssim y$ ($x \gtrsim y$) means that $x\le C y$ ($x\ge C y$) holds for some $C$.


\section{Derivation of the asymptotic models}\label{sec:deri}
In this section, we derive the three asymptotic models (\ref{eq:Boussinesq}), (\ref{wave}), and (\ref{wave:uni}) of system \eqref{eq:1} by means of a multi-scale expansion (cf.  \cite{CL14,CL15,CLS12,CGSW18,granero2018asymptotic,GS}). 
\subsection{The non-local Boussinesq model}
Using (\ref{adg1}), the
system \eqref{eq:1} can be equivalently written as
\begin{subequations}\label{eq:2}
\begin{align}
N_t+(UN)_x+U_x&=0,\\
U_t+UU_x+\frac{(1+B)B_x}{1+N}&=0,\\
B-N-\left(\frac{B_x}{1+N}\right)_x&=0.
\end{align}
\end{subequations}
In this new variables, the initial data \eqref{initialdata} takes the form
\begin{equation}\label{initialdata2}
N(x,0)=n_{0}(x)-1,\;U(x,0)=u_{0}(x).
\end{equation}
Furthermore, we can rewrite (\ref{eq:2}b) as
$$
U_t+UU_x+B_x+BB_x+U_tN+NUU_x=0.
$$
Similarly, (\ref{eq:2}c) can be expanded
$$
B-N-\frac{B_{xx}}{1+N}+\frac{B_x}{(1+N)^2}N_x=0,
$$
and then it takes the similar form
$$
B-N+BN^2-N^3+2NB-2N^2-B_{xx}-NB_{xx}+B_xN_x=0.
$$
Then, \eqref{eq:2} becomes
\begin{subequations}\label{eq:3}
\begin{align}
N_t+U_x&=-(NU)_x,\\
U_t+B_x&=-UU_x-BB_x-U_tN-NUU_x,\\
B-N-B_{xx}&=-BN^2+N^3-2NB+2N^2+NB_{xx}-B_xN_x.
\end{align}
\end{subequations}
Now, from the ansatz (\ref{adg2})
and equating in powers of $\epsilon$, the system \eqref{eq:3} leads to a cascade of linear equations for the coefficients $N^{(\ell)},  B^{(\ell)}$, and $U^{(\ell)}$. The first terms satisfy
\begin{subequations}\label{eq:case0}
\begin{align}
N^{(0)}_t+U^{(0)}_x&=0,\\
U^{(0)}_t+B^{(0)}_x&=0,\\
B^{(0)}-N^{(0)}-B^{(0)}_{xx}&=0,
\end{align}
\end{subequations}
with initial data from (\ref{initialdata2}). System \eqref{eq:case0} can be explicitly decoupled and we find
\begin{subequations}\label{eq:case0b}
\begin{align}
N^{(0)}_t+U^{(0)}_x&=0,\\
U^{(0)}_t+\partial_x(1-\partial_x^2)^{-1}N^{(0)}&=0,\\
B^{(0)}&=(1-\partial_x^2)^{-1}N^{(0)}.
\end{align}
\end{subequations}
Then, (\ref{eq:case0}a), (\ref{eq:case0}b) lead to
\begin{equation}\label{adg5}
N^{(0)}_{tt}+\mathscr{L}N^{(0)}=0,\;
U^{(0)}_{tt}+\mathscr{L}U^{(0)}=0.
\end{equation}
The second term in the expansion solves
\begin{subequations}\label{eq:case1}
\begin{align}
N^{(1)}_t+U^{(1)}_x&=-(N^{(0)}U^{(0)})_x,\\
U^{(1)}_t+B^{(1)}_x&=-U^{(0)}U^{(0)}_x-B^{(0)}B^{(0)}_x-U^{(0)}_tN^{(0)},\\
B^{(1)}-N^{(1)}-B^{(1)}_{xx}&=-2N^{(0)}B^{(0)}+2(N^{(0)})^2+N^{(0)}B_{xx}^{(0)}-B_x^{(0)}N_x^{(0)}.
\end{align}
\end{subequations}
Note that, using (\ref{eq:case0}c), the equation (\ref{eq:case1}c) can be written as
\begin{align*}
B^{(1)}-N^{(1)}-B^{(1)}_{xx}&=-2N^{(0)}(B^{(0)}-N^{(0)}-B_{xx}^{(0)})-N^{(0)}B_{xx}^{(0)}-B_x^{(0)}N_x^{(0)}\\
&=-(N^{(0)}B_{x}^{(0)})_x.
\end{align*}
Then, from (\ref{eq:case0b}c), we have
\begin{equation}\label{adg6}
B^{(1)}=(1-\partial_x^2)^{-1}N^{(1)}-\mathscr{N}(N^{(0)}B_{x}^{(0)})=(1-\partial_x^2)^{-1}N^{(1)}-\mathscr{N}(N^{(0)}\mathscr{N}N^{(0)}).
\end{equation}
Now, using (\ref{eq:case0b}b-c), equation (\ref{eq:case1}b) can be written as
\[
U^{(1)}_t+B^{(1)}_x=-U^{(0)}U^{(0)}_x-B^{(0)}(\mathscr{N}N^{(0)})+(\mathscr{N}N^{(0)})N^{(0)}.
\]
Furthermore, from (\ref{eq:case0b}b), note that the last two terms can be written as
\[-B^{(0)}(\mathscr{N}N^{(0)})+(\mathscr{N}N^{(0)})N^{(0)}=-\mathscr{N}N^{(0)}B^{(0)}_{xx},  \]
and using (\ref{eq:case0b}c) again we obtain
\begin{equation}\label{adg6b}
U^{(1)}_t+B^{(1)}_x=-U^{(0)}U^{(0)}_x+\mathscr{N}N^{(0)}\mathscr{L}N^{(0)}.
\end{equation}

Substitution of $B^{(1)}$ from (\ref{adg6}) into (\ref{adg6b}) leads to
\begin{align}
U^{(1)}_t+\mathscr{N}N^{(1)}&=-\mathscr{L}(N^{(0)}\mathscr{N}N^{(0)})-U^{(0)}U^{(0)}_x+\mathscr{N}N^{(0)}\mathscr{L}N^{(0)}\nonumber \\
&=-U^{(0)}U^{(0)}_x-\left[\mathscr{L},\mathscr{N}N^{(0)}\right]N^{(0)}, \label{eq:U1t}
\end{align}
where the commutator is given by (\ref{adg4b}).  The approximate model \eqref{eq:Boussinesq} for the truncations $h=\varepsilon N^{(0)}+\varepsilon^2 N^{(1)},\;v=\varepsilon U^{(0)}+\varepsilon^2 U^{(1)},$ is derived from (\ref{eq:U1t}) after neglecting $\mathcal{O}(\varepsilon^3)$ terms.

\subsection{The non-local single wave equation model}\label{sec:nonlocalwave}
From (\ref{adg5}) we observe that $N^{(0)}$ and $U^{(0)}$ satisfy the same linear nonlocal wave equation. Thus, under the extra assumption of having the same initial data, we can conclude that
\begin{equation}\label{samedata}
U^{(0)}=N^{(0)}.
\end{equation}
This will allow to further simplify the system \eqref{eq:Boussinesq}. Taking the time derivative of  (\ref{eq:case1}a) and using \eqref{eq:U1t} we find that
\begin{align*}
N^{(1)}_{tt}-\mathscr{N}N^{(1)}_{x}&=(U^{(0)}U^{(0)}_{x})_{x}-\left[\mathscr{L},\mathscr{N}N^{(0)}\right]N^{(0)}_{x}-(N^{(0)}U^{(0)})_{xt}.
\end{align*}
Furthermore, from \eqref{samedata} and \eqref{adg3}, we conclude that
\begin{align*}
N^{(1)}_{tt}+\mathscr{L}N^{(1)}&=\left(N^{(0)}N^{(0)}_x+\left[\mathscr{L},\mathscr{N}N^{(0)}\right]N^{(0)}\right)_x-2\left(N^{(0)}N^{(0)}_t\right)_x.
\end{align*}
Considering now the truncation $h=\varepsilon N^{(0)}+\varepsilon^2 N^{(1)},$ and neglecting contributions of order $\mathcal{O}(\varepsilon^3)$ in the last expression, the single wave equation \eqref{wave} emerges.

\subsection{The unidirectional non local wave model}
In this section, we derive  the unidirectional asymptotic model  \eqref{wave:uni}. We introduce the following far field variables
\begin{equation}\label{far:var}
\chi=x-t, \quad \tau=\varepsilon t.
\end{equation}
Using the chain rule we have that
\begin{equation*}
\frac{\partial^{2}}{\partial t^{2}}h(\chi(x,t),\tau(t))=h_{\chi\chi}-\varepsilon h_{\chi\tau}-\varepsilon h_{\tau\chi}+\varepsilon^{2}h_{\tau\tau}.
\end{equation*}
On the other hand, using the representation \eqref{repre:Q} of the Helmholtz operator $\mathscr{Q}$, it is not hard to see that
$\mathscr{Q}=(1-\partial_{\chi \chi})^{-1}.$
Therefore, from the change of variables \eqref{far:var} and neglecting terms of order $\mathcal{O}(\varepsilon^3)$ (notice that by construction $h\sim \mathcal{O}(\varepsilon)$), we find that equation \eqref{wave} becomes
\begin{equation*}
\left(h_{\chi}-2\varepsilon h_{\tau}\right)_{\chi}+(\mathscr{N}h)_{\chi}=\left( 3 h h_{\chi}-[\mathscr{L},\mathscr{N}h]h\right)_{\chi}
\end{equation*}
which after integrating in $\chi$, reordering terms and going back by abuse of notation to  variables $x$ and $t$ we obtain \eqref{wave:uni}.
\section{Conserved quantities}\label{sec:quantities}
In this section we derive some conserved quantities of the models above. We start with the system (\ref{eq:Boussinesq}). Note first that we can write
\begin{equation}\label{eq:brac}
\left[\mathscr{L},\mathscr{N}h\right]h=\mathscr{L}\left(h\mathscr{N}h\right)+\frac{1}{2}\partial_{x}\left(\mathscr{N}h\right)^{2}.
\end{equation}
Property (\ref{eq:brac}) leads to the formulation of (\ref{eq:Boussinesq})
in conservation form
\begin{equation*}
\partial_{t}\begin{pmatrix}h\\ v\end{pmatrix}+\partial_{x}f(h,v)=0,
\end{equation*}
where 
\begin{equation*}
f(h,v)=\begin{pmatrix} v(1+h)\\\frac{v^{2}}{2}-\mathscr{N}\left(h\mathscr{N}h\right)+\frac{1}{2}\left(\mathscr{N}h\right)^{2}+\mathscr{Q}h\end{pmatrix}.
\end{equation*}
Then, if $u=\mathscr{Q}h$ and we assume that $h,v,u,u_{x}\rightarrow 0$ as $|x|\rightarrow\infty$, we obtain the preservation of
$$I_{1}(h,v)=\int_{\RR}hdx,\quad I_{2}(h,v)=\int_{\RR}vdx.$$
On the other hand, the following lemma is used below to derive a third conserved quantity.
\begin{lemma}
\label{lemm1}
If $h\rightarrow 0$ as $|x|\rightarrow\infty$, then:
\begin{subequations}\label{eq:l1}
\begin{align}
&\int_{\RR}h\mathscr{L}\left(h\mathscr{N}h\right)dx=-\frac{1}{2}\int_{\RR}h\partial_{x}\left(\mathscr{N}h\right)^{2}dx ,\nonumber\\
&\int_{\RR}h\mathscr{N}hdx=0.
\end{align}
\end{subequations}
\end{lemma}
\begin{proof}
We use the Fourier symbols of the operators $\mathscr{L}, \mathscr{N}$, the relation $\partial_{x}\mathscr{N}=-\mathscr{L}$, and Plancherel identity to have
the following identities:
\begin{eqnarray*}
\int_{\RR}h\mathscr{L}\left(h\mathscr{N}h\right)dx&=&\int_{\RR}h\mathscr{N}h\mathscr{L}h dx=-\int_{\RR}h\mathscr{N}h\partial_{x}\mathscr{N}h dx=-\frac{1}{2}\int_{\RR}h\partial_{x}\left(\mathscr{N}h\right)^{2}dx,\\
\int_{\RR}h\mathscr{N}hdx&=&-\int_{\RR}\left(\mathscr{N}h\right)hdx.
\end{eqnarray*}
\end{proof}
\begin{prop}
Let $h, v$ be solutions of (\ref{eq:Boussinesq}) with $h,v\rightarrow 0$ as $|x|\rightarrow\infty$ and let
\begin{equation}\label{eq:conserved1}
I(h,v)=\int_{\RR}hvdx=\int_{\RR} (uv+u_{x}v_{x})dx,
\end{equation}
where $u=\mathscr{Q}h$. Then
$$\frac{d}{dt} I(h,v)=0.$$
\end{prop}
\begin{proof}
Using (\ref{eq:brac}), we write (\ref{eq:Boussinesq}) in the form
\begin{subequations}\label{eq:Boussinesq1}
\begin{align}
h_t+(hv)_x+v_x&=0,\\
v_t+vv_x+\mathscr{L}\left(h\mathscr{N}h\right)+\frac{1}{2}\partial_{x}\left(\mathscr{N}h\right)^{2}+\mathscr{N}h&=0.
\end{align}
\end{subequations}
Multiplying (\ref{eq:Boussinesq1}a) by $v$,  (\ref{eq:Boussinesq1}b) by $h$, adding these two amounts, using Lemma \ref{lemm1}, and the hypotheses on $h$ and $v$, we have
$$
\int_{\RR}\left(h_{t}v+v_{t}h+v((1+h)v)_{x}+vv_{x}h\right)dx=0.
$$ Finally, using that $h,v\rightarrow 0$ as $|x|\rightarrow\infty$ again, note that
\begin{subequations}
\begin{align}
\int_{\RR}\left(v((1+h)v)_{x}+vv_{x}h\right)dx&= \int_{\RR}\left(vv_{x}+v(hv)_{x}+vv_{x}h\right)dx,\nonumber\\
&=\int_{\RR}\left(\frac{\partial}{\partial_{x}}\left(\frac{v^{2}}{2}\right)-v_{x}hv+v_{x}hv\right)dx=0.\nonumber
\end{align}
\end{subequations}
\end{proof}
A final result on (\ref{eq:Boussinesq}) is the Hamiltonian formulation. The proof is direct.
\begin{theorem}
The system (\ref{eq:Boussinesq}) admits a Hamiltonian structure
\begin{equation*}
\partial_{t}\begin{pmatrix}h\\ v\end{pmatrix}=\mathscr{J}\delta E(h,v),
\end{equation*}
where the solution pair $(h,v)$ is smooth enough and vanishes at infinity,
$$\delta E=\left(\frac{\delta E}{\delta h},\frac{\delta E}{\delta v}\right)^{T},$$ denotes the variational derivative,
$$\mathscr{J}=-\partial_{x}\begin{pmatrix}0&1\\ 1&0\end{pmatrix},$$
and
\begin{equation*}
E(h,v)=\frac{1}{2}\int_{\RR}\left(v^{2}(1+h)+(\mathscr{B}h)^{2}+h(\mathscr{N}h)^{2}\right)dx,\quad \mathscr{Q}=\mathscr{B}^{2}.
\end{equation*}
\end{theorem}
On the other hand, using (\ref{eq:brac}), the bidirectional model \eqref{wave} can be written in a conservation form
\begin{equation*}
\partial_{t}(h_{t}+\partial_{x}h^{2})-\partial_{x}\left(\mathscr{N}h+\partial_{x}\left(\frac{h^{2}}{2}\right)-\partial_{x}\mathscr{N}(h\mathscr{N}h)+\partial_{x}\left(\frac{(\mathscr{N}h)^{2}}{2}\right)\right)=0,
\end{equation*} 
which, assuming that $h$ is sufficiently smooth and that $h_{x}$ vanishes at infinity, implies that
$$\frac{d}{dt}\int_{\mathbb{R}}(h_{t}+\partial_{x}h^{2})dx=0.$$ 

As far as the unidirectional model  \eqref{wave:uni} is concerned, using again (\ref{eq:brac}), the model is written in conservation form
\begin{equation}\label{adg8}
h_{t}+\partial_{x}\left(\frac{3}{4}h^{2}+\frac{1}{2}\mathscr{N}(h\mathscr{N}h)-\frac{1}{4}(\mathscr{N}h)^{2}-\frac{1}{2}\mathscr{Q}h-\frac{h}{2}\right)=0,
\end{equation}
which implies, when $h\rightarrow 0$ as $x\rightarrow\pm\infty$, the preservation in time of
$$\int_{\mathbb{R}}hdx.$$ If, in addition, we multiply (\ref{adg8}) by h, integrate on $\mathbb{R}$ and use Lemma \ref{lemm1}, then the $L^{2}$ norm
$$\int_{\mathbb{R}}h^{2}dx,$$ is the second conserved quantity. Finally, 
the unidirectional model  \eqref{wave:uni}
also admits a Hamiltonian structure
\begin{equation*}
h_{t}=\frac{1}{2\epsilon}\partial_{x}\delta E(h),
\end{equation*}
where now $\delta E=\frac{\delta E}{\delta h}$ and
\begin{equation}\label{adg9}
E(h)=\frac{1}{2}\int_{\RR}\left(h^{2}-h^{3}+(\mathscr{B}h)^{2}+h(\mathscr{N}h)^{2}\right)dx,
\end{equation}
and where the phase space for (\ref{adg9}) involves smooth enough functions $h$ vanishing at infinity.
\section{Well-posedness for the non-local Boussinesq system}\label{sec:nonlocalBouss}
In this section we will show the well-posedness of system \eqref{eq:Boussinesq}. Due to the coupled nature of the equations when performing the a priori energy estimates we need to \textit{symmetrize} the system. To this end, let us introduce the following functional space 
\begin{equation}\label{adg12}
\mathcal{X}=\{(h,v)\in H^{2}(\RR)\times H^{3}(\RR): \norm{(h,v)}_{\mathcal{X}}=\norm{(h,v)}_{L^2(\RR)\times L^{2}(\RR)}^{2}+\norm{\sqrt{\mathscr{L}}\pa_{x}^{2}h}_{L^2(\RR)}^{2} +  \norm{\pa_{x}^{3}v}_{L^2(\RR)}^{2}<\infty \}.
\end{equation}
If $m(\xi)=\frac{\xi^2}{1+\xi^2}$ denotes  the Fourier multiplier of the operator $\mathscr{L}$ (cf. (\ref{adg4})) then in (\ref{adg12}) $\mathscr{T}=\sqrt{\mathscr{L}}$ denotes the operator with Fourier symbol $\sqrt{m(\xi)}$, and therefore it formally satisfies $\mathscr{T}^{2}={\mathscr{L}}$. Then the norm introduced in the definition of $\mathcal{X}$ in (\ref{adg12}) is related to a classical seminorm in ${H}^{k}(\RR)$ as follows.
%
%
%

\begin{lemma}\label{lema:LtoL2}
Let  $k\in \mathbb{N}$ and $f\in L^{2}(\RR)$ be smooth enough. Then, there exists a constant $C>0$ such that 
\begin{equation}\label{estimate:lema:LtoL2}
\norm{\pa_{x}^{k}f}_{L^{2}} \leq C\left(\norm{f}_{L^{2}}+\norm{\sqrt{\mathscr{L}}\pa_{x}^{k}f}_{L^{2}}\right).
\end{equation}
\end{lemma}
\begin{proof}
Let $R>0$ and ${B(0,R)}=\{ x\in\mathbb{R}: |x|\leq R\}$. 
Using Parseval identity we have that
\[
\norm{\pa_{x}^{k}f}_{L^{2}}^{2}=\norm{\widehat{\pa_{x}^{k}f}}_{L^2}^{2}=\int_{\RR} \xi^{2k}|\hat{f}(\xi)|^{2}=\int_{B(0,R)} \xi^{2k}|\hat{f}(\xi)|^{2} \  d\xi +\int_{\mathbb{R}\setminus B(0,R)} \xi^{2k}|\hat{f}(\xi)|^{2} \  d\xi .
\]
The first integral can be bounded  by
\begin{equation}\label{esti1}
\int_{B(0,R)} \xi^{2k}|\hat{f}(\xi)|^{2} \  d\xi \leq R^{2k}\int_{\mathbb{R}} |\hat{f}(\xi)|^{2} \  d\xi = R^{2k}\norm{f}_{L^{2}}^{2}.
\end{equation}
On the other hand, note that for $\xi\in \mathbb{R}\setminus B(0,R)$ we have the pointwise bound
\[\frac{\xi^{2}}{1+\xi^{2}}\geq \frac{R^{2}}{1+R^{2}}.\]
Then it holds that
\begin{equation}\label{est2}
\int_{\mathbb{R}\setminus B(0,R)} \xi^{2k}|\hat{f}(\xi)|^{2} \  d\xi \leq  \frac{1+R^{2}}{R^{2}} \int_{\mathbb{R}\setminus B(0,R)}  \frac{\xi^{2+k}}{1+\xi^{2}}\xi^{k}  |\hat{f}(\xi)|^{2} \  d\xi \leq  \frac{1+R^{2}}{R^{2}} \norm{\sqrt{\mathscr{L}}\pa_{x}^{k}f}_{L^{2}}^{2}.
\end{equation}

Therefore, choosing for instance $R=1$, (\ref{esti1}), (\ref{est2}) yield \eqref{estimate:lema:LtoL2}.
\end{proof}

Then we see that $\mathcal{X}$ is a modified version of $H^2\times H^3$.
\begin{theorem}\label{theoremBouss}
For $(h_0,v_0)\in H^2\times H^3$  there exist a time $0<T_{max}$ and a unique solution 
\[ (h,v)\in C((0,T_{max}), H^2\times H^3)\]
of the ivp of \eqref{eq:Boussinesq} with $h(x,0)=h_{0}(x), v(x,0)=v_{0}(x), x\in\mathbb{R}$.
\end{theorem}

\begin{proof}
We first focus on obtaining some a priori estimates.
We define the energy
\begin{equation}\label{energy:def}
\mathcal{E}(t)=\norm{(h,v)}_{L^{2}\times L^2}^{2} 	+\norm{\sqrt{\mathscr{L}}\pa_{x}^{2}h}_{L^{2}}^{2}+ \norm{\pa_{x}^{3}v}_{L^{2}}^{2},
\end{equation}
where the norm considered in the space $L^{2}(\mathbb{R})\times L^{2}(\mathbb{R})$ is given by
$$\norm{(f,g)}_{L^{2}\times L^2}=\left(||f||_{L^{2}}^{2}+||g||_{L^{2}}^{2}\right)^{1/2}, f,g\in L^{2}(\mathbb{R}).$$
Multiplying the first equation of \eqref{eq:Boussinesq} by $h$, the second by $v$, adding the resulting equalities and integrating on $\mathbb{R}$ we have
\begin{align*}
\frac{1}{2}\frac{d}{dt}\norm{(h,v)}_{L^{2}\times L^2}^{2}=-\int_{\RR} \left((hv)_{x}+v_{x}\right) h \  dx- \int_{\RR} \left(vv_{x}+[\mathscr{L}, \mathscr{N}h]h+\mathscr{N}h \right) v \  dx.
\end{align*}
Now integration by parts, the application of H\"older and Young  inequalities, and the Sobolev embedding $H^{\frac{1}{2}+\epsilon}(\RR)\hookrightarrow L^{\infty}(\RR)$ for $\epsilon>0$ lead to 
\begin{align}\label{L2:estimate:system3}
\frac{1}{2}\frac{d}{dt}\norm{(h,v)}_{L^{2}\times L^2}^{2}\leq C \norm{v_x}_{L^{\infty}}\norm{h}^{2}_{L^{2}} \leq \norm{\pa_{x}^{3}v}^{3}_{L^2}+\norm{h}_{L^{2}}^{3}.
\end{align}
Next, we deal with the other terms in \eqref{energy:def}. Multiplying the first equation in \eqref{eq:Boussinesq} by $ \mathscr{L}\partial_{x}^{4}h$ and integrating we have 
\begin{align}\label{eq:Lh}
\int_{\RR} \mathscr{L}\pa_{x}^{4}h h_{t} \  dx=-\int_{\RR}(hv)_{x} \mathscr{L}\pa_{x}^{4}h \ dx - \int_{\RR} v_{x} \mathscr{L}\pa_{x}^{4}h \  dx.
\end{align}
On the other hand, multiplying the second equation in  \eqref{eq:Boussinesq} by $\pa_{x}^{6}v$ and integrating we obtain
\begin{align}\label{eq:H3v}
\int_{\RR} v_{t} \pa_{x}^{6}v\  dx &=-\int_{\RR}vv_{x} \pa_{x}^{6}v \ dx - \int_{\RR}[\mathscr{L}, \mathscr{N}h] h \pa_{x}^{6}v  \ dx - \int_{\RR} \mathscr{N}h \pa_{x}^{6} v \  dx .
\end{align}
Since $-\mathscr{L}=\pa_{x}\mathscr{N}$, integrating by parts we find that the last term in \eqref{eq:Lh} is given by
\begin{equation}\label{rewrittingI2}
 - \int_{\RR} v_{x} \mathscr{L}\pa_{x}^{4}h \  dx=\int_{\RR} v_{x} \mathscr{N}\pa_{x}^{5}h \  dx= -\int_{\RR} \pa_{x}^{6} v \mathscr{N}h \ dx .   
 \end{equation}
Moreover, we have that
\begin{equation}\label{rewritingtemporal}
\int_{\RR} h_{t}\mathscr{L}\pa_{x}^{4}h  \  dx=\frac{1}{2}\frac{d}{dt}\norm{\sqrt{\mathscr{L}}\pa_{x}^{2}h} ^{2}_{L^{2}}, \quad  -\int_{\RR} v_{t} \pa_{x}^{6}v\  dx=  \frac{1}{2}\frac{d}{dt}\norm{\pa_{x}^{3}v}^{2}_{L^{2}}.
\end{equation}
Then, adding \eqref{eq:Lh} and \eqref{eq:H3v}, and using \eqref{rewrittingI2},\eqref{rewritingtemporal} lead to
\begin{equation}\label{adg13}
\frac{1}{2}\frac{d}{dt}\left( \norm{\sqrt{\mathscr{L}}\pa_{x}^{2}h}_{L^{2}}^{2}+ \norm{\pa_{x}^{3}v}_{L^{2}}^{2} \right) =\underbrace{ -\int_{\RR}(hv)_{x} \mathscr{L}\pa_{x}^{4}h \ dx}_{I_{1}}+\underbrace{\int_{\RR}vv_{x} \pa_{x}^{6}v \ dx}_{I_{2}}+\underbrace{\int_{\RR}[\mathscr{L}, \mathscr{N}h] h \pa_{x}^{6}v  \ dx}_{I_{3}}.
\end{equation}
We now estimate each of the integrals $I_{i}$ in (\ref{adg13}).
First, notice that integration by parts yields 
\begin{equation*}
I_{2}=-\int_{\RR} \pa_{x}^{2}(vv_{x})\pa_{x}^{3}v  \  dx =-\frac{3}{2}\int_{\RR}(\pa_{x}^{3}v)^{2} v_{x} \  dx,
\end{equation*}
and thus
\begin{align}\label{est:I2}
|I_{2}|\leq C  \norm{\pa_{x}v}_{L^{\infty}}\norm{\pa_{x}^{3}v}^{2}_{L^{2}}\leq  C \norm{\pa_{x}^{3}v}^{3}_{L^{2}},
\end{align}
where in the second inequality we used the Sobolev embedding $H^{\frac{1}{2}+\epsilon}(\RR)\hookrightarrow L^{\infty}(\RR)$ for $\epsilon>0.$

Integrating by parts the first term $I_{1}$ we find that
\begin{equation*}
I_{1}=-\int_{\RR}\pa_{x}^{3}(hv)\pa_{x}^{2}\mathscr{L}h \ dx =\underbrace{-\int_{\RR}\left( \pa_{x}^{3}v h+ \pa_{x}^{2}vh_{x}+v_{x}\pa_{x}^{2}h \right)\pa_{x}^{2}\mathscr{L}h \ dx}_{J_{1}}\underbrace{-\int_{\RR}v\pa_{x}^{3}h \pa_{x}^{2}\mathscr{L}h \ dx}_{J_{2}} .
\end{equation*}
We use H\"older inequality to estimate $J_{1}$ as
\begin{align}\label{est:J1}
|J_{1}|&\leq C \left( \norm{h}_{L^{\infty}}\norm{\pa_{x}^{3}v}_{L^{2}}+\norm{\pa_{x}^{2}v}_{L^{\infty}}\norm{\pa_{x}h}_{L^{2}}+\norm{v_{x}}_{L^{\infty}}\norm{h_{xx}}_{L^{2}}\right)\norm{\pa_{x}^{2}\mathscr{L}h}_{L^{2}} \nonumber \\
& \leq C \norm{h_{xx}}_{L^{2}}\norm{\pa_{x}^{3}v}_{L^{2}}\norm{\pa_{x}^{2}\mathscr{L}h}_{L^{2}}.
\end{align}
As for $J_{2}$, using the identity (cf. \eqref{comp:helmholtz})
\begin{equation}\label{adg14}
\mathscr{L}=-\partial_{x}^2 (1-\pa_{x}^{2})^{-1}= \mathrm{Id}-\mathscr{Q},
\end{equation}
we obtain that
\begin{equation*}\label{rewrite:J2}
J_{2}=-\int_{\RR}v \pa_{x}^{3}h \left(\pa_{x}^{2}h-\mathscr{Q}\pa_{x}^{2}h\right) \ dx =- \frac{1}{2}\int_{\RR}v \pa_{x}(\pa_{x}^{2}h)^{2} \ dx + \int_{\RR}v \pa_{x}^{3}h \mathscr{Q}\pa_{x}^{2}h \ dx.
\end{equation*}
Using integration by parts in both terms we infer that
\begin{equation}\label{est:J2}
|J_{2}|\leq C \norm{v_{x}}_{L^{\infty}} \norm{h_{xx}}^{2}_{L^{2}}\leq \norm{\pa_{x}^{3}v}_{L^{2}} \norm{h_{xx}}^{2}_{L^{2}}.
\end{equation}
In a similar way, expanding the commutator and using (\ref{adg14}) again lead to
\begin{align*}
I_{3}=\int_{\RR} \left( \mathscr{L}(\mathscr{N}h h)-\mathscr{N}h\mathscr{L}h\right)\pa_{x}^{6}v \  dx=\underbrace{- \int_{\RR} \mathscr{Q}(\mathscr{N}h h)\pa_{x}^{6}v \  dx}_{K_{1}} +\underbrace{ \int_{\RR} \mathscr{N}h  \mathscr{Q}h \pa_{x}^{6}v \  dx}_{K_{2}}.
\end{align*}
We rewrite both terms $K_{1}$ and $K_{2}$ as
\begin{align*}
K_{1}= -\int_{\RR} (\mathscr{N}hh)_{x} \mathscr{L}\pa_{x}^{3}v \  dx,
\end{align*}
and
\begin{align*}
K_{2}=-\int_{\RR} (\mathscr{N}h\mathscr{Q}h)_{xxx} \pa_{x}^{3}v \  dx &=- \frac{1}{2} \int_{\RR} \pa_{x}^{4}\left( (\mathscr{Q}h)^{2}
\right) \pa_{x}^{3}v \  dx =-\int_{\RR}\pa_{x}^{2}\left( (\mathscr{Q}h_{x})^{2}-\mathscr{Q}h\mathscr{L}h\right) \pa_{x}^{3}v \  dx.
\end{align*}
Using the fact that, cf. (\ref{adg4}), 
\begin{equation}\label{adg15}
\norm{\sqrt{\mathscr{L}}f}_{L^2}\leq C \norm{f}_{L^{2}},\; f\in L^{2},
\end{equation}
then the term $K_{1}$ can be bounded by
\begin{equation}\label{bound:K1prev}
|K_{1}|\leq C \norm{h_{xx}}_{L^{2}}^{2}\norm{\mathscr{L}\pa_{x}^{3}v}_{L^{2}} \leq \norm{h_{xx}}_{L^{2}}^{2}\norm{\pa_{x}^{3}v}_{L^{2}},
\end{equation}
for some constant $C$. As far as $K_{2}$ is concerned, by expanding the derivatives, tedious but a straightforward computation and using (\ref{adg15}) again shows that
\begin{equation}\label{bound:K2prev}
|K_2| \leq C \norm{\mathscr{Q}h}_{L^{\infty}} \norm{\pa_{x}^{3}v}_{L^{2}}\norm{\sqrt{\mathscr{L}}\pa_{x}^2 h}_{L^{2}} \leq C    \norm{\pa_{x}^{3}v}_{L^{2}}\norm{h_{xx}}_{L^{2}} \norm{\sqrt{\mathscr{L}}\pa_{x}^2 h}_{L^{2}},
\end{equation}
for some constant $C$. From \eqref{est:I2}, \eqref{est:J1}, \eqref{est:J2}, \eqref{bound:K1prev}, and \eqref{bound:K2prev} we conclude that

\begin{align}
\frac{1}{2}\frac{d}{dt}\left( \norm{\sqrt{\mathscr{L}}\pa_{x}^{2}h}_{L^{2}}^{2}+ \norm{\pa_{x}^{3}v}_{L^{2}}^{2} \right) &\leq  C \bigg( \norm{\pa_{x}^{3}v}^{3}_{L^{2}}+\norm{\pa_{x}^{3}v}_{L^{2}} \norm{h_{xx}}^{2}_{L^{2}} \nonumber \\
&\hspace{3cm} +\norm{\pa_{x}^{3}v}_{L^{2}}\norm{h_{xx}}_{L^{2}} \norm{\sqrt{\mathscr{L}}\pa_{x}^2 h}_{L^{2}} \bigg).\label{adg16}
\end{align}

Applying Lemma \ref{lema:LtoL2} and Young's inequality to (\ref{adg16}) leads to
\begin{align}\label{H2:estimate:system3}
\frac{1}{2}\frac{d}{dt}\left( \norm{\sqrt{\mathscr{L}}\pa_{x}^{2}h}_{L^{2}}^{2}+ \norm{\pa_{x}^{3}v}_{L^{2}}^{2} \right) \leq  C \left( \norm{\sqrt{\mathscr{L}}\pa_{x}^2 h}^{3}_{L^{2}}+\norm{h}^{3}_{L^{2}} +\norm{\pa_{x}^{3}v}^{3}_{L^{2}}\right),
\end{align}
for some constant $C$. 
Combining estimates \eqref{L2:estimate:system3} and \eqref{H2:estimate:system3} yields
\begin{equation}\label{energy:estimate:bouss:nolocal}
\mathcal{E}'(t)\leq C \mathcal{E}^{\frac{3}{2}}(t),
\end{equation}
which ensures a local time of existence $t^{*}>0$ such that $\mathcal{E}(t)\leq 4 \mathcal{E}(0), 0<t<t^{*}.$
In order to construct the solutions, we first define the approximate problems using mollifiers (cf. proof of Theorem \ref{theorem1}). More precisely, the regularized system is given by
\begin{subequations}
	\begin{align}
		h^{\epsilon}_t+\mathcal{J}_{\epsilon}(\mathcal{J}_{\epsilon}h^{\epsilon}\mathcal{J}_{\epsilon}v^{\epsilon})_x+\mathcal{J}_{\epsilon}\mathcal{J}_{\epsilon}v^{\epsilon}_x&=0, \label{regularized:eq:Boussinesq} \\
		v^{\epsilon}_t+\mathcal{J}_{\epsilon}(\mathcal{J}_{\epsilon}v^{\epsilon}\partial_{x}\mathcal{J}_{\epsilon}v^{\epsilon})+\left[\mathscr{L},\mathscr{N}h^{\epsilon}\right]h^{\epsilon}+\mathscr{N}h^{\epsilon}&=0.\label{regularized:eq:Boussinesq2}
	\end{align}
\end{subequations}

 By the properties of the mollifiers we can repeat the previous energy estimates and provide the same a priori bounds for the regularized system of \eqref{regularized:eq:Boussinesq}-\eqref{regularized:eq:Boussinesq2}. Hence, we will find a uniform time of existence $T_{max}>0$ for the sequence of regularized problems. To conclude the argument, we pass to the limit. Furthermore, the continuity in time for the solution (instead of merely weak continuity) is obtained as follows: first, the energy estimate \eqref{energy:estimate:bouss:nolocal} yields the strong right continuity at $t=0$. Moreover, it is easy to check that changing variables $\tilde{t}=-t$ provides the strong left continuity at $t=0$ and hence the continuity in time for the solutions. To conclude let us remark that the uniqueness follows by a classical contradiction argument as in Theorem \ref{theorem1}.
\end{proof}

\section{Well-posedness in Sobolev spaces for the bidirectional non-local wave equation}\label{sec:bidirec}
In this section, we provide the local-well posedness on the bidirectional non-local wave equation \eqref{wave}.  

\begin{theorem}\label{thn:bid:wave}
Let $(h_0,h_1)$ be such that $(h_0-1,h_1)\in H^4\times H^3$ and
$$
\|h_0-1\|_{L^\infty}<1/2.
$$
Then, there exist $0<T$ and a unique solution to \eqref{wave}
$$
(h-1,h_t)\in C([0,T],H^4\times H^3),
$$ with initial value $(h_0,h_1)$. 
\end{theorem}

\begin{proof}
As before, existence and uniqueness of solutions of  \eqref{wave} are based on deriving useful a priori energy estimates. To this end, we write $h=1+w$ and then equation \eqref{wave} becomes 
\begin{equation}\label{wave:w}
w_{tt}+\mathscr{L}w=\left((1+w)w_x+\left[\mathscr{L},\mathscr{N}w\right]w\right)_{x}-2\left((1+w)w_t\right)_x.
\end{equation}
We define the energy
\begin{equation}\label{def:wave3}
\mathcal{E}(t)=\norm{w_{t}}^{2}_{H^3}+\norm{\sqrt{\mathscr{L}}\pa_{x}^{3}w}^{2}_{L^2}+\norm{w}^{2}_{H^4}.
\end{equation}
Testing equation \eqref{wave:w} against $w_{t}$ and integrating by parts we have
\begin{align*}
\frac{1}{2}\frac{d}{dt}\left(\norm{w_{t}}_{L^{2}}^{2}+\norm{\sqrt{\mathscr{L}}w}_{L^{2}}^{2}\right)&=\int_{\RR}(w_{x}^{2}+(1+w)w_{xx})w_{t}  \ dx +\int_{\RR} \left(\left[\mathscr{L},\mathscr{N}w\right]w\right)_{x} w_{t} \ dx  \nonumber \\
&\quad - 2\int_{\RR}\left( (1+w)w_{tx}w_{t}+w_{x}w_{t}^{2}\right) \ dx\leq C \mathcal{E}(t)^{3/2}.
\end{align*}
In particular,
\begin{equation*}\label{L2:est:wave}
\frac{d}{dt} \norm{w_{t}}_{L^{2}}^{2}\leq C \mathcal{E}(t)^{3/2}.
\end{equation*}
Furthermore, from the Cauchy-Schwarz inequality and \eqref{def:wave3}
\begin{equation*}
\frac{d}{dt}\norm{w}_{L^{2}(\RR)}^{2}=2\int_{\RR}w_{t}w \  dx \leq 2\norm{w}_{L^2}\norm{w_{t}}_{L^2} \leq \mathcal{E}(t).
\end{equation*}
On the other hand, it holds that
\begin{align*}
\frac{1}{2}\frac{d}{dt}\left(\norm{\pa_{x}^{3}w_{t}}_{L^{2}}^{2}+\norm{\sqrt{\mathscr{L}}\pa_{x}^{3}w}_{L^{2}}^{2}\right)&=\underbrace{-\int_{\RR}(w_{x}^{2}+(1+w)w_{xx})\pa_{x}^{6}w_{t}  \ dx}_{M_{1}}\underbrace{-\int_{\RR} \left[\mathscr{L},\mathscr{N}w\right]w \pa_{x}^{7}w_{t} \ dx}_{M_{2}}  \\
&\quad + \underbrace{2\int_{\RR}\left( (1+w)w_{tx}\pa_{x}^{6}w_{t}+w_{x}w_{t}\pa_{x}^{6}w_{t}\right) \ dx}_{M_{3}}.
\end{align*}
In order to estimate each of the integrals $M_{i}$, we first notice a hiding energy extra term in $M_{1}$. Integrating by parts we obtain
\begin{align*}
M_{1}= -\frac{1}{2}\frac{d}{dt}\norm{\pa_{x}^{4}w}_{L^{2}}^{2} \underbrace{-\int_{\mathbb{R}} w_{x}^{2} \pa_{x}^{6}w_{t} \  dx}_{M_{12}}\underbrace{-\int_{\mathbb{R}}w w_{xx} \pa_{x}^{6}w_{t} \ dx}_{M_{13}}.
\end{align*}
It is easy to check that after integration by parts 
\begin{align}\label{M12:estimate}
\left|M_{12}\right| \leq \int_{\RR} \left| \pa_{x}^{3}(\pa_{x}w)^{2} \pa_{x}^{3}w_{t} \ dx \right| \leq C\left(\norm{\pa_{x}^{4}w}_{L^{2}}\norm{w_{x}}_{L^\infty}+\norm{\pa_{x}^{3}w}_{L^2}\norm{w_{xx}}_{L^\infty} \right) \norm{\pa_{x}^{3}w_{t}}_{L^2}.
\end{align}
Similarly, integrating by parts we have that
\begin{align}\label{adg16b}
M_{13}= -\int_{\RR}  \pa_{x}^{3}(ww_{xx}) \pa_{x}^{3}w_{t} \ dx= -\int_{\RR} w \pa_{x}^{5}w \pa_{x}^{3}w_{t} \ dx + \texttt{l.o.t},
\end{align}
where 
\begin{equation}\label{lot}
| \texttt{l.o.t}| \leq C\left(  \norm{\pa_{x}^{3}w}_{L^{2}}\norm{w_{xx}}_{L^{\infty}}+\norm{w_{xx}}_{L^{\infty}}\norm{w_{xx}}_{L^2}+\norm{w_x}_{L^{\infty}}\norm{\pa_{x}^{4}w}_{L^{2}}\right)\norm{\pa_{x}^{3}w_{t}}_{L^{2}},
\end{equation}
and, after integration by parts again
\begin{align}\label{adg17}
-\int_{\RR}w \pa_{x}^{5}w \pa_{x}^{3}w_{t} \  dx=\frac{1}{2}\int_{\RR} w\pa_{t}(\pa_{x}^{4}w)^{2} \ dx+\int_{\RR} w_{x}\pa_{x}^{4}w \pa_{x}^{3}w_{t} \ dx .
\end{align}
The last term in (\ref{adg17}) can be bounded as
\begin{equation}\label{lot2}
\left | \int_{\RR}w_{x}\pa_{x}^{4}w \pa_{x}^{3}w_{t} \  dx \right| \leq C \norm{w_{x}}_{L^{\infty}} \norm{\pa_{x}^{4}w}_{L^{2}}\norm{ \pa_{x}^{3}w_{t}}_{L^2}.
\end{equation}
Thus, from (\ref{adg16b}), estimates \eqref{lot},\eqref{lot2} imply that
\begin{equation}\label{rewriteM13}
M_{13}=\frac{1}{2}\int_{\RR} w \pa_{t} \left( \pa_{x}^{4}w^{2}\right) \ dx +  \mbox{l.o.t}
\end{equation}
where by means of Young's inequality we find that
\begin{align*}
| \text{l.o.t}| \leq C\left(  \norm{\pa_{x}^{3}w}_{L^{2}}\norm{w_{xx}}_{L^{\infty}}+\norm{w_{xx}}_{L^{\infty}}\norm{w_{xx}}_{L^2}+\norm{w_x}_{L^{\infty}}\norm{\pa_{x}^{4}w}_{L^{2}}\right)\norm{\pa_{x}^{3}w_{t}}_{L^{2}} \leq C \mathcal{E}(t)^{3/2}.
\end{align*}
In order to estimate the commutator term $M_{2}$, let us recall that from \eqref{comp:helmholtz} we have
\[ 
M_{2}=-\int_{\RR}\left[\mathscr{Q},\mathscr{Q}w_{x} \right]w \pa_{x}^{7}w_{t} \ dx=- \int_{\RR} \mathscr{Q}\left(\mathscr{Q}w_{x}w\right)\pa_{x}^{7}w_{t} \ dx + \frac{1}{2}\int_{\RR} \pa_{x}\left(\mathscr{Q}w\right)^{2} \pa_{x}^{7}w_{t} \ dx.
\]
Using the duality $\dot{H}^{-4}-\dot{H}^{4}$ argument and the fact that  $\mathscr{Q}$ is continuous between $H^{s}(\RR)$  and $H^{s+2}(\RR)$ for any $s\in \RR$  we readily check that
\begin{align}\label{M2:estimate}
\left| M_{2} \right| &\leq C\left( \norm{\mathscr{Q}(\mathscr{Q}w_{x}w)}_{\dot{H}^{4}}+\norm{\pa_{x}(\mathscr{Q}w)^2}_{\dot{H}^{4}}\right) \norm{\pa_{x}^{7}w_{t}}_{\dot{H}^{-4}}  \nonumber \\
&\leq C\left( \norm{\mathscr{Q}w_{x}w}_{\dot{H}^{2}}+\norm{(\mathscr{Q}w)^2}_{\dot{H}^{5}}\right) \norm{\pa_{x}^{3}w_{t}}_{L^2} \nonumber \\
&\leq C\left( \norm{w_{x}}_{L^{2}}\norm{w_{xx}}_{L^{2}}+\norm{\pa_{x}^{3}w}^{2}_{L^{2}}\right) \norm{\pa_{x}^{3}w_{t}}_{L^2} \leq C \mathcal{E}(t)^{3/2},
\end{align} 
for some constant $C$, where in the last inequality we used that $H^{s}$ is a Banach algebra for $s>\frac{1}{2}$ and Young's inequality.
Similarly, splitting $M_{3}$ and integrating by parts we infer that
\begin{align}\label{M3:estimate}
|M_{3}| &= \left|2\int_{\RR}  w_{tx}\pa_{x}^{6}w_{t}+ww_{tx}\pa_{x}^{6}w_{t}+w_{x}w_{t}\pa_{x}^{6}w_{t} \  dx\right| \nonumber \\
&\leq C\bigg(  \norm{\pa_{x}^{3}w}_{L^{2}}\norm{w_{tx}}_{L^{\infty}}+\norm{w_{xx}}_{L^{\infty}}\norm{w_{txx}}_{L^2}+\norm{w_x}_{L^{\infty}}\norm{\pa_{x}^{3}w_{t}}_{L^{2}}+\norm{\pa_{x}^{4}w}_{L^2}\norm{w_{t}}_{L^{\infty}} \nonumber\\
&\hspace{1cm} + \norm{\pa_{x}^{3}w}_{L^2}\norm{w_{xt}}_{L^{\infty}}\bigg) \times \norm{\pa_{x}^{3}w_{t}}_{L^{2}} \leq C \mathcal{E}(t)^{3/2}
\end{align}
Hence, combining estimates \eqref{M12:estimate}-\eqref{M3:estimate} and from \eqref{def:wave3} we conclude that
\begin{equation*}\label{est:energy1}
\frac{d}{dt} \mathcal{E}(t)\leq C \left( \mathcal{E}(t)^{3/2}+\mathcal{E}(t)\right)+\frac{1}{2}\int_{\RR} w \pa_{t}\left( \pa_{x}^{4}w\right) \  dx.
\end{equation*}
Integrating in time leads to
\begin{equation*}\label{adg18}
\mathcal{E}(t)\leq \mathcal{E}(0)+C\int_{0}^{t}\left(\mathcal{E}^{3/2}(s)+\mathcal{E}(s)\right) \ ds  +\frac{1}{2}\int_{0}^{t}\int_{\RR} w(x,s) \pa_{s} \left( \pa_{x}^{4} w(x,s)\right)^{2} \ dx ds.
\end{equation*}
To deal with the latter integral we use Fubini's theorem and  integrate by parts in time which yields
\begin{align}\label{energy:afterparts}
\mathcal{E}(t)&\leq \mathcal{E}(0)+C\int_{0}^{t}\left(\mathcal{E}^{3/2}(s)+\mathcal{E}(s)\right) \ ds -\frac{1}{2}\int_{0}^{t}\int_{\RR} w_{s}(x,s)  (\pa_{x}^{4} w(x,s))^{2} \ dx ds \nonumber \\
&\hspace{2cm}-\frac{1}{2}\left(\int_{\RR} w_{0}(x)(\pa_{x}^{4}w_{0}(x))^{2} \ dx-\int_{\RR}w(x,t) (\pa_{x}^{4}w(x,t))^{2} \  dx\right).
\end{align}
Defining, for $T>0$
\[ \mathsf{E}_{T}=\displaystyle\sup_{0\leq t \leq T} \mathcal{E}(t)= \displaystyle\sup_{0\leq t \leq T} \left(\norm{w_{t}}^{2}_{H^3}+\norm{\sqrt{\mathscr{L}}\pa_{x}^{3}w}^{2}_{L^2}+\norm{w}^{2}_{H^4}\right), \]
taking the supremum in time in \eqref{energy:afterparts} and H\"older's inequality lead to
\begin{align}\label{est:energy2}
	\mathsf{E}_{T}&\leq  \mathsf{E}_{0}+Ct\left(\mathsf{E}_{T}^{3/2}+	\mathsf{E}_{T}\right)  +\frac{1}{2}\norm{w_{0}}_{L^{\infty}}\norm{\pa_{x}^{4}w_{0}}_{L^{2}}^{2} +\frac{1}{2}\displaystyle\sup_{0\leq t \leq T}\left(\norm{w}_{L^{\infty}}\norm{\pa_{x}^{4}w}_{L^{2}}^{2}\right) \nonumber \\ 
	&\hspace{4cm}+  \frac{1}{2}\displaystyle\sup_{0\leq t \leq T}\int_{0}^{t} \norm{w_{s}}_{L^{\infty}}  \norm{\pa_{x}^{4} w}_{L^{2}}^{2}  ds.
\end{align}

Furthermore, writing
\[
w(x,t)=w_{0}(x)+w(x,t)-w_{0}(x)=w_{0}(x)+\int_{0}^{t} \partial_{s}w(x,s) \ ds,	
\]
then Sobolev embedding yields 
\begin{equation}\label{Linftybound}\displaystyle\sup_{0\leq t \leq T} \norm{w}_{L^{\infty}}\leq \norm{w_0}_{L^{\infty}}+t \displaystyle\sup_{0\leq t \leq T}\norm{\pa_{x}^{3}w_{t}}_{L^{2}}. 
	\end{equation}
Therefore combining \eqref{est:energy2}, \eqref{Linftybound}, and Young's inequality we find that
\begin{align}\label{est:energy3}
	\mathsf{E}_{T}&\leq  \mathsf{E}_{0}+\mathsf{E}^{3/2}_{0}+Ct\left(\mathsf{E}_{T}^{3/2}+	\mathsf{E}_{T}\right) +\norm{w_0}_{L^{\infty}}\displaystyle\sup_{0\leq t \leq T} \norm{\pa_{x}^{4}w}_{L^{2}}^{2} ,
\end{align}
for some constant $C$. Now taking $w_{0}$ with $\norm{w_0}_{L^{\infty}}=\frac{1}{2}$, from the definition of $E_{T}$ and \eqref{est:energy3} it holds that
\begin{align}
	\mathsf{E}_{T}&\leq 2 \left(\mathsf{E}_{0}+\mathsf{E}^{3/2}_{0}\right)+2Ct\left(\mathsf{E}_{T}^{3/2}+\mathsf{E}_{T}\right),\nonumber
\end{align}
which is valid for $t\in (0,t_{1}), t_{1}= \displaystyle\min\{1,T\}$. (Sharper inequalities can be obtained, if necessary, from smaller choices of $\norm{w_0}_{L^{\infty}})$.

Noticing that  $\mathsf{E}^{q}_{T}\leq \mathsf{E}^{2}_{T}+1$ for $q=1$ and $q=\frac{3}{2}$ and the fact that $t\in (0,t_{1}), t_{1}= \displaystyle\min\{1,T\}$ (in particular $0<t< 1$), we find that the polynomial estimate
\begin{align}\label{est:energy5}
	\mathsf{E}_{T}&\leq\mathsf{N}_{0}+4Ct\mathsf{E}_{T}^{2},
\end{align}
holds. Here we have used the notation $\mathsf{N}_{0}=2 \left(\mathsf{E}_{0}+\mathsf{E}^{3/2}_{0}+2C\right)$. Similar polynomial estimates as \eqref{est:energy5} have been derived in \cite{Coutand-Shkoller-06}. Let us define the polynomial $\mathcal{P}_{t}(y)= 4Cty^2-y+\mathsf{N_{0}}$, so that $\mathcal{P}_{t}(\mathsf{E_{T}})\geq 0$. The roots of $\mathcal{P}_{t}$ are computed as
\[ y_{\pm}= \frac{1 \pm \sqrt{1-16Ct\mathsf{N}_{0}}}{8Ct}.\]

\begin{figure}[h!]
\label{picture}
\includegraphics[width=0.6\textwidth]{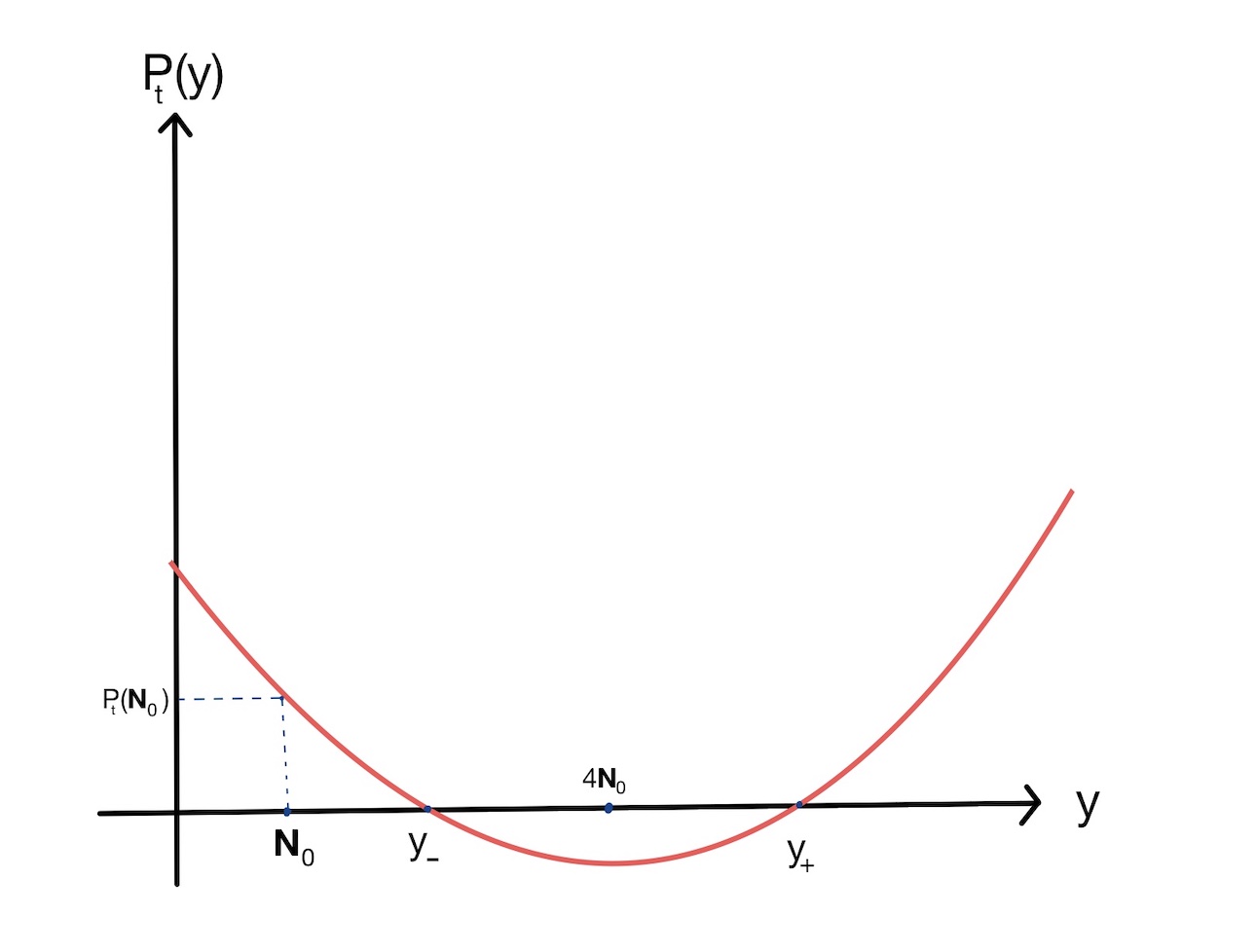}
\centering
\caption{Positive real roots of the polynomial $P_{t}(y)$ for $t\ll 1$.}
\end{figure}
Taking $0<t\ll 1$ (for instance $0<t=\frac{1}{32C\mathsf{N}_{0}}$) we have (cf. Figure 1)
\[ 0<y_{-}=4\mathsf{N_{0}}\left( 1 - \sqrt{1/2} \right) <  y_{+}=4\mathsf{N_{0}}\left( 1 + \sqrt{1/2} \right)\]
and 
\[0<y_{-}< 4\mathsf{N}_{0}< y_{+}.\]

Furthermore, the application $t\to\mathsf{E}_{t}$ is continuous for $t\in [0,T)$. This, together with the fact that
$\mathsf{E}_{0}<\mathsf{N}_{0}$ and $\mathcal{P}_{t}(\mathsf{N}_{0})>0$, implies that 
\begin{equation}\label{est:energy6}
	\forall t\in  (0, t^{\star}), \quad  \mathsf{E}_{T}< 4 \mathsf{N}_{0},  
	\end{equation}
where $t^{\star}=\displaystyle\min\{\frac{1}{32C\mathsf{N}_{0}},1,T \}$.

Similarly as before, in order to perform the a priori estimates rigorously and construct the local existence of solutions we follow  a regularization procedure. More precisely, the regularized version of \eqref{wave:w}  takes the form

\begin{equation}\label{regularized:wave:w}
w^{\epsilon}_{tt}+\mathscr{L}w^{\epsilon}=\left(\mathcal{J}^{\epsilon}\left((1+\mathcal{J}^{\epsilon}w^{\epsilon})\mathcal{J}^{\epsilon}w^{\epsilon}_x\right)+\left[\mathscr{L},\mathscr{N}w^{\epsilon}\right]w^{\epsilon}\right)_{x}-2\mathcal{J}^{\epsilon}\left((1+\mathcal{J}^{\epsilon}w^{\epsilon})\mathcal{J}^{\epsilon}w^{\epsilon}_t\right)_x.
\end{equation}
Repeating the same a priori estimates for the regularized  equation \eqref{regularized:wave:w} we can deduce the analogue of \eqref{est:energy6} namely
 \begin{equation}\label{est:energy7}
 	\forall t\in  (0, t^{\star}), \quad  \mathsf{E}_{T}< 4 \mathsf{N}_{0},  
 \end{equation}
where $t^{\star}=\displaystyle\min\{\frac{1}{32C\mathsf{N}_{0}},1,T^{\epsilon} \}$. Notice that the time of existence of the solution depends a priori on the regularization parameter $\epsilon$. Nevertheless, the following argument shows that \eqref{est:energy7} holds independently of the parameter $\epsilon$. Indeed, take $T^{\epsilon}$ the first time such that 
\[ \mathsf{E}_{T^{\epsilon}}=8\mathsf{N}_{0}.\]
We observe that the precise choice of the quantity $8\mathsf{N}_{0}$ is not special. One could choose also any other quantity bigger than $ 4\mathsf{N}_{0}$ to make the argument work.
If the previous equality does not hold, that is, $\mathsf{E}_{T^{\epsilon}}<8\mathsf{N}_{0}$, this implies that $T^{\epsilon}=\infty,$ and hence we conclude (since $t^{\star}$ is the minimum) that $t^\star$ is independent of $\epsilon$. On the other hand if $T^{\epsilon}$ is finite then $t^{\star}=\displaystyle\min\{\frac{1}{32C\mathsf{N}_{0}},1,T^{\epsilon} \}=\displaystyle\min\{\frac{1}{32C\mathsf{N}_{0}},1\}$. Indeed, this assertion follows from  invoking the continuity of $\mathsf{E}_{T}$ of the regularized problem, and using that by construction $\mathsf{E}_{T^{\epsilon}}=8\mathsf{N}_{0}$ and $\mathsf{E}_{t^{\star}}<4\mathsf{N}_{0}$. This concludes the proof. 
%
\end{proof}

\section{Well-posedness in Sobolev spaces for unidirectional non-local wave model}\label{sec:uni:wp}
The main goal of this section is to prove the local well-posedness of equation \eqref{wave:uni} in $L^2$ based Sobolev spaces. Note first that, using \eqref{comp:helmholtz}, we have
$ -\mathscr{L}=\mathscr{Q}-1,$
and \eqref{wave:uni} can be rewritten as (we set $\varepsilon=1$ for simplicity)

\begin{equation}\label{final:eq:uni:2}
h_{t}=-\frac{1}{2}\left(3 hh_{x}+[\mathscr{Q},\mathscr{Q}h_{x}]h - \mathscr{Q}h_{x}-h_{x} \right). 
\end{equation}
For this equation, we show the following local well-posedness result.
\begin{theorem}\label{theorem1}
Let $s>\frac{3}{2}$  and $h_0\in H^{s}(\RR)$ {with mean zero}. Then there exist $T_{\max}>0$ and a unique local solution $h\in C([0,T_{\max}),H^{s}(\RR))$ of \eqref{final:eq:uni:2} with $h(0)=h_{0}$.
\end{theorem}

\begin{proof}
The proof follows from the combination of appropriate a priori energy estimates and the use of a suitable approximation procedure using mollifiers, cf. \cite{Majda-Bertozzi-2001}. Thus, we first focus in deriving a priori energy estimates and later comment briefly on the approximation procedure to construct the solution. We begin by reminding (cf. Section \ref{sec:quantities}) that the mean property is conserved in time by the solutions
\[ \int_{\RR} h(x,t) \ dx=\int_{\RR} h_0(x) \ dx, \]
as well as the $L^{2}$ norm
\begin{equation*}
 \int_{\RR} h^{2}(x,t) \ dx=\int_{\RR} h^{2}_0(x) \ dx.
 \end{equation*}
Furthermore, applying $\Lambda^{s}$ to (\ref{final:eq:uni:2}), multiplying by $\Lambda^{s}h$ and integrating we obtain

\begin{align}
\frac{1}{2}\frac{d}{dt}\norm{h}_{\dot{H}^{s}}^{2}&=-\frac{3}{2}\int_{\RR} \Lambda^{s}(hh_{x})\Lambda^{s}h \ dx-\frac{1}{2}\int_{\RR}\Lambda^{s}\left([\mathscr{Q},\mathscr{Q}_{x}h]h\right)\Lambda^{s}h \ dx \label{Hs:eq:1} \\
&\hspace{1cm}+\frac{1}{2}\int_{\RR} \Lambda^{s}\mathscr{Q}h_{x} \Lambda^{s}h \ dx +\frac{1}{2}\int _{\RR}\Lambda^{s}h_{x} \Lambda^{s}h\ dx. \nonumber
\end{align}
Using the self-adjointness of the operator $\mathscr{Q}$, a straightforward computation shows that the last two terms on the right hand-side of \eqref{Hs:eq:1} are zero. The first term can be estimated by means of the classical Kato-Ponce commutator as follows. Using integration by parts, we rewrite 
\begin{align*}
\int_{\RR} \Lambda^{s}(hh_{x})\Lambda^{s}h \ dx&= \int_{\RR} [\Lambda^{s},h]h_{x}\Lambda^{s}h \ dx + \int_{\RR} h \Lambda^{s}h_{x}\Lambda^{s}h \ dx \\
&= \int_{\RR} [\Lambda^{s},h]h_{x}\Lambda^{s}h \ dx-\frac{1}{2} \int_{\RR} h_{x} |\Lambda^{s}h|^{2} \ dx.
\end{align*}
Now, invoking the first equation in Lemma \ref{Kato-Ponce commutator estimate} with $p=2, p_{1}=\infty, p_{2}=2$, we have that
\begin{align}
\av{\int_{\RR} \Lambda^{s}(hh_{x})\Lambda^{s}h \ dx} &\lesssim  \norm{[\Lambda^{s},h]h_{x}}_{L^{2}}\norm{\Lambda^{s}h}_{L^{2}}+ \norm{h_{x}}_{L^\infty}\norm{\Lambda^{s}h}^{2}_{L^{2}} \nonumber \\
&\lesssim  \norm{h_{x}}_{L^\infty}\norm{\Lambda^{s}h}^{2}_{L^{2}}. \label{Hs:est:1}
\end{align}
The second term on the right hand side in \eqref{Hs:eq:1} can be bounded as follows.
Expanding the commutator and using the self-adjointness of $\mathscr{Q}$ yield
\begin{align}
\int_{\RR}\Lambda^{s}\left([\mathscr{Q},\mathscr{Q}_{x}h]h\right)\Lambda^{s}h \ dx &= \int_{\RR} \Lambda^{s}(\mathscr{Q}h_{x}h)\Lambda^{s}\mathscr{Q}h \ dx -\int_{\RR} \Lambda^{s}(\mathscr{Q}h_{x}\mathscr{Q}h)\Lambda^{s}h \ dx.
\end{align}
Then applying the second estimate in Lemma \ref{Kato-Ponce commutator estimate} with $p=2, p_{1},p_{3}=\infty, p_{2},p_{4}=2$, the Sobolev embedding  $H^{\frac{1}{2}+\epsilon}(\RR)\hookrightarrow L^{\infty}(\RR)$ for $\epsilon>0$, and the fact that  $\mathscr{Q}$ is continuous between $H^{s}(\RR)$  and $H^{s+2}(\RR)$ for any $s\in \RR$  we find that
\begin{align*}
 \av{\int_{\RR} \Lambda^{s}(\mathscr{Q}h_{x}h)\Lambda^{s}\mathscr{Q}h \ dx} &\lesssim \norm{\Lambda^{s}(\mathscr{Q}h_{x}h)}_{L^{2}}\norm{\Lambda^{s}\mathscr{Q}h}_{L^{2}} \\
&\lesssim \left( \norm{\Lambda^{s}\mathscr{Q}h_{x}}_{L^{\infty}}\norm{h}_{L^{2}}+\norm{\mathscr{Q}h_{x}}_{L^{\infty}}\norm{\Lambda^{s}h}_{L^{2}} \right)\norm{\Lambda^{s}\mathscr{Q}h}_{L^{2}}
 \\
&\lesssim \left( \norm{h}_{H^{s-\frac{1}{2}+\epsilon}}\norm{h}_{L^{2}}+\norm{h}_{H^{-\frac{1}{2}+\epsilon}}\norm{h}_{H^{s}}\right) \norm{h}_{H^{s-2}}.
\end{align*}
Similarly, one can show that
\[ \av{\int_{\RR} \Lambda^{s}(\mathscr{Q}h_{x}\mathscr{Q}h)\Lambda^{s}h \ dx} \lesssim \left( \norm{h}_{H^{s-\frac{1}{2}+\epsilon}}\norm{h}_{H^{-2}}+\norm{h}_{H^{-\frac{1}{2}+\epsilon}}\norm{h}_{H^{s-2}}\right) \norm{h}_{H^{s}}.
\]
Thus, 
\begin{align}\label{Hs:est:2}
\av{\int_{\RR}\Lambda^{s}\left([\mathscr{Q},\mathscr{Q}_{x}h]h\right)\Lambda^{s}h \ dx} \leq C\norm{h_{0}}_{L^{2}} \left(\norm{h}_{H^{s-\frac{1}{2}+\epsilon}}+\norm{h}_{H^{s-2}}\right) \norm{h}_{H^{s}}
\end{align} 
where we have used that $\norm{h}_{H^{-\frac{1}{2}+\epsilon}}+\norm{h}_{H^{-2}}\leq C \norm{h}_{L^{2}}=C\norm{h_0}_{L^{2}}$, for some constant $C$.

Therefore, combining estimates \eqref{Hs:est:1}, \eqref{Hs:est:2} leads to
\begin{equation}\label{Hs:est:prev:3}
\frac{d}{dt}\norm{h}_{H^{s}}^{2}\leq C \norm{h_{x}}_{L^{\infty}}  \norm{h}_{H^{s}}^{2}+C \norm{h_{0}}_{L^{2}} \left(\norm{h}_{H^{s-\frac{1}{2}+\epsilon}}+\norm{h}_{H^{s-2}}\right) \norm{h}_{H^{s}}
\end{equation}
which in particular gives the following inequality
\begin{equation}\label{Hs:est:prev:4}
\frac{d}{dt}\norm{h}_{H^{s}}^{2}\leq C \left(\norm{h_{x}}_{L^{\infty}}+ 1 \right)  \norm{h}_{H^{s}}^{2}.
\end{equation}
Using the Sobolev embedding  $H^{\frac{1}{2}+\epsilon}(\RR)\hookrightarrow L^{\infty}(\RR)$, $\epsilon>0$, we observe that
\begin{equation}\label{Hs:est:3}
\frac{d}{dt}\norm{h}_{H^{s}}^{2}\leq C \norm{h}_{H^{s}}^{3},
\end{equation}
for some constant $C$ independent of $t$.
Defining $\mathcal{E}(t)=\norm{h}_{H^{s}}^{2}$, estimate \eqref{Hs:est:3} leads to the following differential equation
\begin{equation}\label{dif:eq:energy:uni}
	 \mathcal{E}'(t)\leq C\mathcal{E}^{3/2}(t), 
  \end{equation}
with $C=C_{s}>0$  which ensures a uniform time of existence $T_{\max}>0$ such that
\[ \mathcal{E}(t)\leq 4\mathcal{E}(0), 0<t<T_{\max}.\]
Once this uniform time of existence has been obtained, the local existence result follows classical regularization procedure (cf. \cite{Majda-Bertozzi-2001}) as follows. First, we consider a symmetric and positive mollifier $\mathcal{J}\in C^{\infty}_{c}$, $\mathcal{J}(x)=J(|x|)$ such that $\int_{\RR}
\mathcal{J}=1$. For $\epsilon>0$ we define $\mathcal{J}^{\epsilon}= \frac{1}{\epsilon}\mathcal{J}(\frac{x}{\epsilon})$ and consider the regularized problem
\begin{align}\label{reg:final:eq}
\pa_{t}h^{\epsilon}=-\frac{1}{2}\big(\mathcal{J}^{\epsilon}(\mathcal{J}^{\epsilon}h \partial_{x}\mathcal{J}^{\epsilon}h) + [\mathscr{Q},\mathscr{Q}h^{\epsilon}_{x}]h^{\epsilon} - \mathscr{Q}h^{\epsilon}_{x}-\mathcal{J}^{\epsilon}\mathcal{J}^{\epsilon}h^{\epsilon}_{x} \big).
\end{align}
In \eqref{reg:final:eq}, the conserved quantities are also preserved by the flow  and the previous bounds \eqref{Hs:est:1}-\eqref{Hs:est:2} and thus \eqref{Hs:est:prev:3} hold. Therefore, we may find a time of existence $T^{\star}>0$ for the sequence of regularized problems. 
Using  compactness arguments and passing to the limit conclude the proof of existence.  
The continuity in time for the solution (instead of merely weak continuity) is obtained by classical arguments (cf. \cite{Majda-Bertozzi-2001}): On the one hand, the differential equation \eqref{dif:eq:energy:uni} yields the strong right continuity at $t=0$. On the other hand, it is easy to check that changing variables $\tilde{t}=-t$, we can repeat once again the same bounds and provide the strong left continuity at $t=0$. Combining both arguments shows the continuity in time of the solution.

As for uniqueness, let $h^{1},h^{2}\in C([0,T_{\max}),H^{s}(\RR))$ be two solutions corresponding to the same initial condition and denote $\widehat{h}=h^{2}-h^{1}$. Then $\widehat{h}$ satisfies
\begin{equation}\label{adg10}
\pa_{t}\widehat{h}=-\frac{1}{2}\left(3 h^{2}h^{2}_{x}-3 h^{1}h^{1}_{x}+[\mathscr{Q},\mathscr{Q}h^{2}_{x}]h^{2}-[\mathscr{Q},\mathscr{Q}h^{1}_{x}]h^{1} - \mathscr{Q}\widehat{h}_{x}-\widehat{h}_{x}  \right).
\end{equation}
Then, multiplying (\ref{adg10}) by  $\widehat{h}$ and integrating,
we have
\begin{align*}
\frac{1}{2}\frac{d}{dt}\norm{\widehat{h}}_{L^{2}}^{2}=-\frac{3}{2} \int_{\RR} \left( \widehat{h}h^{2}_{x}+h^{1}\widehat{h}_{x}\right) \widehat{h} \ dx+ \int_{\RR} \bigg[ [\mathscr{Q},\mathscr{Q}h^{2}_{x}]h^{2}-[\mathscr{Q},\mathscr{Q}h^{1}_{x}]h^{1}\bigg] \widehat{h} \  dx.
\end{align*}
Using (\ref{eq:l1}b), it is not hard to check that each commutator can be rewritten as a Burgers-type nonlinear-term and therefore
\begin{align*}
 \av{\int_{\RR} \bigg[ [\mathscr{Q},\mathscr{Q}h^{2}_{x}]h^{2}-[\mathscr{Q},\mathscr{Q}h^{1}_{x}]h^{1}\bigg] \widehat{h} \  dx } &\leq C \left( \norm{h^{1}}_{L^{\infty}}+\norm{h^{2}}_{L^{\infty}}\right)\norm{\widehat{h}}_{L^{2}}^{2}\leq C,
 \end{align*}
for some constant $C$.
Similarly, integrating by parts, 
\[ \av{\int_{\RR} \left( \widehat{h}h^{2}_{x}+h^{1}\widehat{h}_{x}\right) \widehat{h} \ dx } \leq C \left( \norm{\pa_{x}h^{1}}_{L^{\infty}}+\norm{\pa_{x}h^{2}}_{L^{\infty}}\right)\norm{\widehat{h}}_{L^{2}}^{2}.\]
Defining $\beta(t)=\norm{h^{1}}_{L^{\infty}}+\norm{h^{2}}_{L^{\infty}}+\norm{\pa_{x}h^{1}}_{L^{\infty}}+\norm{\pa_{x}h^{2}}_{L^{\infty}}$, then
\[\frac{1}{2}\frac{d}{dt}\norm{\widehat{h}}_{L^{2}}^{2} \leq C \beta(t)\norm{\widehat{h}}_{L^{2}}^{2}. \] 
Uniqueness holds applying Gr\"onwall's inequality.
\end{proof}
\begin{rem}
One can readily check that a direct consequence of the derived estimates in the proof of Theorem \ref{theorem1} provides the following blow-up criterion when combined with the logarithmic Sobolev inequality (cf. Lemma \ref{Sob:log}) which we state as a theorem for the sake of clarity.
\begin{theorem}[Blow-up criteria]\label{theorem2}
Let $s>\frac{3}{2}$, 
$h_0\in H^{s}(\RR)$ with zero mean and let $T_{\max}>0$ be the lifespan associated to the solution $h$ to \eqref{final:eq:uni:2} with $h(x,0)=h_0(x)$. Then $h$ blows-up in finite time $T_{\max}$ if and only if
\begin{equation}\label{blowu}
\int_{0}^{T_{\max}}\norm{h_{x}(\tau)}_{\text{BMO}} \  d\tau=\infty.
\end{equation}
\end{theorem}
\begin{proof}
The result follows as a direct consequence of estimate \eqref{Hs:est:prev:4}. Using the logarithmic inequality in Lemma \ref{Sob:log} we find that
\begin{equation}\label{adg11}
\frac{d}{dt}\norm{h}_{H^{s}}\leq C \left(1+\norm{h_{x}}_{\text{BMO}}\big[1+\displaystyle\log(1+\norm{h}_{H^{s}}) \big]  \right)  \norm{h}_{H^{s}},
\end{equation}
for some constant $C$. From (\ref{adg11}), it holds that
\begin{equation*}
1+\norm{h}_{H^{s}}\leq \bigg[(1+ \norm{h_0}_{H^{s}})\mbox{exp}\left(1+CT_{\max} \right)\bigg]^{\mbox{exp}\left(C\int_{0}^{T_{\max}}(\norm{h_{x}(\tau)}_{\text{BMO}}\ d\tau\right)}.
\end{equation*}
Therefore, if there exists $L>0$ such that  $$\displaystyle\int_{0}^{T_{\max}} \norm{h_{x}(\tau)}_{\text{BMO}} \  d\tau <L,$$ then
\begin{equation*}
1+\norm{h}_{H^{s}}\leq \bigg[(1+ \norm{h_0}_{H^{s}})\mbox{exp}\left(1+CT_{\max} \right)\bigg]^{\mbox{exp}\left(CL\right)}.
\end{equation*}
To show the converse, if $\int_{0}^{T_{\max}}\norm{h_{x}(t)}_{\text{BMO}} \  dt=\infty$, by means of the embedding $H^{\frac{1}{2}}(\RR)\subset \text{BMO}(\RR)$, we deduce that the solution $h(x,t)$ will blow up in finite time and \eqref{blowu} follows.
\end{proof}
\end{rem}

\section{A wave breaking result for the unidirectional non-local wave model}\label{sec:breaking}
In this section, we investigate the possible wave breaking phenomena for equation \eqref{wave:uni}, that is, the formation of an infinite slope in the solution in the $x$-direction. As we did for the well-posedness result in Section \ref{sec:uni:wp}, it is more convenient to work with the alternative formulation \eqref{final:eq:uni:2}. 
The following lemma shows that for the maximal time of existence $T_{\max}>0$, we have that the solutions remains bounded. More precisely, 
\begin{lem}\label{Lem:bound:Linf}
Let $s>{7/2}$, $h_{0}\in H^{s}(\mathbb{R})$, and let $T_{max}>0$ be the maximal time of existence of the unique solution $h$ 
of \eqref{final:eq:uni:2} given by Theorem \ref{theorem1}. Then 
\begin{equation}\label{bound:Linf}
\displaystyle\sup_{t\in [0,T_{\max})}\norm{h(t)}_{L^{\infty}(\RR)} <\infty.
\end{equation}
\end{lem}
\begin{proof}
To establish the $L^{\infty}$ bound, we follow a pointwise method (cf. \cite{Constantin-Escher-98, Cordoba-Cordoba-04}). {Due to Theorem \ref{theorem1}, we have that $h \in C([0,T_{\max}), H^{s})\cap C^{1}([0, T_{\max}),H^{s-1})$, hence by the Sobolev embedding theorem, if $s>7/2$ we have that $h \in C^{1}([0,T_{\max})\times \RR)$.} In particular, 
\begin{align*}
m(t)=\displaystyle\inf_{x\in\RR}h(x,t)=h(\underline{x}_{t},t), \quad M(t)= \displaystyle\sup_{x\in\RR}h(x,t)=h(\overline{x}_{t},t) , \mbox{ for } t>0,
\end{align*}
(for some $\underline{x}_{t}, \overline{x}_{t}$)
are Lipschitz functions. Following \cite{ Cordoba-Cordoba-04}, one can readily check that $M(t)$ satisfies
\begin{align*}
\left|M(t)-M(s)\right|&=\left\{\begin{array}{cc}h(\overline{x}_t,t)-h(\overline{x}_s,s) \text{ if } M(t)>M(s)\\
h(\overline{x}_s,s)-h(\overline{x}_t,t) \text{ if }M(s)>M(t)\end{array}\right.
\\
&\leq \left\{\begin{array}{cc}h(\overline{x}_t,t)-h(\overline{x}_t,s) \text{ if } M(t)>M(s)\\
h(\overline{x}_s,s)-h(\overline{x}_s,t) \text{ if }M(s)>M(t)\end{array}\right.
\\
&\leq \left\{\begin{array}{cc}|\partial_t h(\overline{x}_t,z)||t-s| \text{ if } M(t)>M(s)\\
|\partial_t h(\overline{x}_s,z)||t-s| \text{ if }M(s)>M(t)\end{array}\right.
\\
&\leq \max_{y,z}|\partial_t h(y,z)||t-s|.
\end{align*}
Similarly
$$
\left|m(t)-m(s)\right|\leq \max_{y,z}|\partial_t h(y,z)||t-s|.
$$
From Rademacher's theorem it holds that $M(t)$ and $m(t)$ are differentiable in $t$ almost everywhere. Furthermore
\begin{align*}
M'(t)&=\lim_{\delta\rightarrow0} \frac{h(\overline{x}_{t+\delta},t+\delta)-h(\overline{x}_t,t)}{\delta}=\lim_{\delta\rightarrow0} \frac{h(\overline{x}_{t+\delta},t+\delta)-h(\overline{x}_t,t)\pm h(\overline{x}_{t+\delta},t)}{\delta}\\
&\leq\lim_{\delta\rightarrow0} \frac{h(\overline{x}_{t+\delta},t+\delta)-h(\overline{x}_{t+\delta},t)}{\delta}\leq \partial_t h(\overline{x}_{t},t).
\end{align*}
In a similar fashion, we obtain that
\begin{align*}
M'(t)&=\lim_{\delta\rightarrow0} \frac{h(\overline{x}_{t+\delta},t+\delta)-h(\overline{x}_t,t)}{\delta}=\lim_{\delta\rightarrow0} \frac{h(\overline{x}_{t+\delta},t+\delta)-h(\overline{x}_t,t)\pm h(\overline{x}_{t},t+\delta)}{\delta}\\
&\geq\lim_{\delta\rightarrow0} \frac{h(\overline{x}_{t},t+\delta)-h(\overline{x}_{t},t)}{\delta}\geq \partial_t h(\overline{x}_{t},t).
\end{align*}
As a consequence
\begin{equation}\label{adg19}
M'(t)=\partial_t h(\overline{x}_{t},t) \text{ a.e.}
\end{equation}
Similarly
$$
m'(t)=\partial_t h(\underline{x}_{t},t) \text{ a.e.}
$$
Therefore, from (\ref{adg19}), evaluating \eqref{final:eq:uni:2} at $x=\overline{x}_{t}$ and noticing that $h_{x}(\overline{x}_{t},t)=0$, we have
\begin{align}\label{adg20}
M'(t)=-\frac{1}{2}\left([\mathscr{Q},\mathscr{Q}h_{x}(\overline{x}_{t})]h(\overline{x}_{t}) - \mathscr{Q}h_{x}(\overline{x}_{t})\right). 
\end{align}
Moreover, from \eqref{repre:Q}, the estimates 
\begin{equation*}
\norm{G}_{L^{2}}=\frac{1}{2}, \quad \norm{\pa_{x}G}_{L^{2}}= \frac{1}{2},
\end{equation*}

and Young's  inequality, it holds that
\begin{align}
\mathscr{Q}h(x_{t}) &\leq \norm{G}_{L^{2}(\RR)}\norm{h}_{L^{2}(\RR)}= \frac{1}{2} \norm{h}_{L^{2}(\RR)},\label{adg21b}\\
[\mathscr{Q},\mathscr{Q}h_{x}(x_{t})]h(x_{t}) &\leq \norm{G}_{L^{2}(\RR)}\norm{ \mathscr{Q}h_{x} h}_{L^{2}(\RR)}+\norm{\mathscr{Q}h_{x}}_{L^{\infty}(\RR)} \norm{\mathscr{Q}h}_{L^{\infty}(\RR)} \leq \frac{1}{2} \norm{h}_{L^{2}}^{2}.\label{adg21c}
\end{align}
From (\ref{adg20}), the estimates (\ref{adg21b}), (\ref{adg21c}) and the preservation of the $L^{2}$ norm by the solutions of \eqref{final:eq:uni:2} imply that
\begin{equation*}
M'(t) \leq \frac{1}{4} \left( 1+\norm{h_{0}}_{L^{2}}^{2}\right),
\end{equation*}
and therefore
\begin{equation}\label{adg22}
M(t) \leq M(0)+\frac{1}{4} \left( 1+\norm{h_{0}}_{L^{2}}^{2}\right)t.
\end{equation}
for $t\in [0,T_{\max})$.
In a similar way, we have 
\begin{equation}\label{adg23}
m(t) \geq m(0)- \frac{1}{4} \left( 1+\norm{h_{0}}_{L^{2}}^{2}\right)t.
\end{equation}
for $t\in [0,T_{\max})$. Thus combining (\ref{adg22}) and (\ref{adg23}) we conclude that
\begin{equation}\label{L:infty:est:fin}
\displaystyle\sup_{t\in[0,T_{\max})} \norm{h(t)}_{L^{\infty}(\RR)} \leq \norm{h_0}_{L^{\infty}(\RR)}+\frac{1}{4}\left( 1+\norm{h_{0}}_{L^{2}}^{2}\right)T_{\max}<\infty,
\end{equation}
and \eqref{bound:Linf} holds.
\end{proof}
Next, let us state the wave breaking result. 
\begin{theorem}\label{theorem3}
Let $ s> {\frac{9}{2}}, h_0\in H^{s}(\mathbb{R})$, and let $h$ be the solution of \eqref{final:eq:uni:2} with initial value $h_0$. Assume that 
\begin{equation}\label{cond:ini}
\inf_{x\in\RR}h_{0,x}(x)\leq -H_{0},
\end{equation}
for some positive constant $H_{0}$, which depends on $\norm{h_0}_{L^{2}(\RR)},\|h_0\|_{L^\infty}$ and is specified below. Then there exists $T_{b}<\infty$ such that
\begin{equation}\label{waveb}
\displaystyle\lim \inf_{t\to T_{b}}\left(\inf_{x\in\RR} h_{x}(x,t)\right)=-\infty.
\end{equation} 
\end{theorem}
\begin{proof}
Similarly as in the proof of Lemma \ref{Lem:bound:Linf}, we use a pointwise argument in order to derive an ODE which breaks down in finite time. By Theorem \ref{theorem1}, we have
\[
h\in C([0,T_{\max}),H^s)\cap C^1([0,T_{\max}),H^{s-1}),
\]
and therefore, by Sobolev embedding, if \(s>9/2\), then
\[
h_x\in C^1([0,T_{\max})\times \RR).
\]
We define
\[
m(t)=\inf_{x\in\RR}h_x(x,t)=h_x(\underline{x}_t,t), \qquad t\in[0,T_{\max}),
\]
for some point \(\underline{x}_t\in\RR\). Arguing as in Lemma \ref{Lem:bound:Linf}, one sees that \(m(t)\) is Lipschitz, and hence, by Rademacher's theorem,
\[
m'(t)=\partial_t h_x(\underline{x}_t,t), \qquad \text{for a.e. } t\in[0,T_{\max}).
\]
Differentiating \eqref{final:eq:uni:2} with respect to \(x\) and using \eqref{adg14}, we obtain
\begin{equation}\label{eq:der:m:new}
h_{xt}
=
-\frac12\Big(3h_x^2+3hh_{xx}+\partial_x[\mathscr Q,\mathscr Q h_x]h-\mathscr Q h+h-h_{xx}\Big).
\end{equation}
Evaluating \eqref{eq:der:m:new} at \(x=\underline{x}_t\), and using that \(h_{xx}(\underline{x}_t,t)=0\), we arrive at
\begin{equation}\label{evol:min:new}
m'(t)=-\frac32\,m(t)^2+\mathcal K(t),
\end{equation}
where
\[
\mathcal K(t)
=
\frac12\Big(\partial_x[\mathscr Q,\mathscr Q h_x]h(\underline{x}_t,t)-\mathscr Q h(\underline{x}_t,t)+h(\underline{x}_t,t)\Big).
\]

We next estimate \(\mathcal K(t)\). First, using \eqref{repre:Q}, we have
\begin{equation}\label{est:K1:new}
\|\mathscr Q h+h\|_{L^\infty}
\le
\|G\ast h\|_{L^\infty}+\|h\|_{L^\infty}
\le
\|G\|_{L^2}\|h\|_{L^2}+\|h\|_{L^\infty}
=
\frac12\|h\|_{L^2}+\|h\|_{L^\infty}.
\end{equation}

For the commutator term, using \eqref{comp:helmholtz} we obtain
\begin{align*}
\partial_x[\mathscr Q,\mathscr Q h_x]h(\underline{x}_t,t)
&=
\mathscr Q\big(((\mathscr Q-I)h)h\big)(\underline{x}_t,t)
-((\mathscr Q-I)h)(\underline{x}_t,t)\,\mathscr Q h(\underline{x}_t,t) \\
&\quad
+\mathscr Q\big((\mathscr Q h_x)h_x\big)(\underline{x}_t,t)
-(\mathscr Q h_x(\underline{x}_t,t))^2.
\end{align*}
We now rewrite the last two terms. Using the kernel representation \eqref{repre:Q}, we have
\begin{align*}
&\mathscr Q\big((\mathscr Q h_x)h_x\big)(\underline{x}_t,t)
-(\mathscr Q h_x(\underline{x}_t,t))^2 \\
&\qquad=
\int_{\RR}G(\underline{x}_t-y)\,\mathscr Q h_x(y,t)\,h_x(y,t)\,dy
-\mathscr Q h_x(\underline{x}_t,t)\int_{\RR}G(\underline{x}_t-y)h_x(y,t)\,dy \\
&\qquad=
\int_{\RR}G(\underline{x}_t-y)\big(\mathscr Q h_x(y,t)-\mathscr Q h_x(\underline{x}_t,t)\big)h_x(y,t)\,dy.
\end{align*}
Integrating by parts in \(y\), we conclude that
\[
\mathscr Q\big((\mathscr Q h_x)h_x\big)(\underline{x}_t,t)
-(\mathscr Q h_x(\underline{x}_t,t))^2
=
\int_{\RR}G'(\underline{x}_t-y)\big(\mathscr Q h_x(y,t)-\mathscr Q h_x(\underline{x}_t,t)\big)h(y,t)\,dy
-\mathscr Q\big(((\mathscr Q-I)h)h\big)(\underline{x}_t,t).
\]
Substituting this identity into the previous expression, we infer that
\begin{align*}
\big|\partial_x[\mathscr Q,\mathscr Q h_x]h(\underline{x}_t,t)\big|
&\le
\left|
\int_{\RR}G'(\underline{x}_t-y)\big(\mathscr Q h_x(y,t)-\mathscr Q h_x(\underline{x}_t,t)\big)h(y,t)\,dy
\right| \\
&\quad
+\|(\mathscr Q-I)h\|_{L^\infty}\|\mathscr Q h\|_{L^\infty}.
\end{align*}

Now, by the mean value theorem and \eqref{comp:helmholtz},
\[
|\mathscr Q h_x(y,t)-\mathscr Q h_x(\underline{x}_t,t)|
\le
|y-\underline{x}_t|\,\|(\mathscr Q-I)h\|_{L^\infty},
\]
and therefore
\[
\big|\partial_x[\mathscr Q,\mathscr Q h_x]h(\underline{x}_t,t)\big|
\le
\Big(\||x|G'\|_{L^2}\|h\|_{L^2}+\|\mathscr Q h\|_{L^\infty}\Big)\|(\mathscr Q-I)h\|_{L^\infty}.
\]
Using again Young's inequality, we have
\[
\|\mathscr Q h\|_{L^\infty}
=
\|G\ast h\|_{L^\infty}
\le
\|G\|_{L^2}\|h\|_{L^2}
=
\frac12\|h\|_{L^2},
\]
and
\[
\|(\mathscr Q-I)h\|_{L^\infty}
\le
\|\mathscr Q h\|_{L^\infty}+\|h\|_{L^\infty}
\le
\frac12\|h\|_{L^2}+\|h\|_{L^\infty}.
\]
Therefore,
\begin{equation}\label{est:K4:new}
\big|\partial_x[\mathscr Q,\mathscr Q h_x]h(\underline{x}_t,t)\big|
\le
C\big(\|h\|_{L^2}^2+\|h\|_{L^\infty}\|h\|_{L^2}\big).
\end{equation}

Combining \eqref{est:K1:new}, \eqref{est:K4:new}, the \(L^\infty\) estimate \eqref{L:infty:est:fin}, and the conservation of the \(L^2\)-norm, we deduce that
\begin{align}
\mathcal K(t)
&\le
C\Big(\|h\|_{L^2(\RR)}^2+\|h\|_{L^2(\RR)}+\|h\|_{L^\infty(\RR)}\Big) \nonumber\\
&\le
C\Big(\|h_0\|_{L^\infty(\RR)}+\big(1+\|h_0\|_{L^2(\RR)}^2\big)(t+1)\Big).
\label{est:K:fin:new}
\end{align}
Inserting \eqref{est:K:fin:new} into \eqref{evol:min:new}, we obtain
\begin{equation}\label{evol:min2:new}
m'(t)\le -\frac32\,m(t)^2+A+Bt,
\end{equation}
where \(A,B>0\) depend only on \(\|h_0\|_{L^\infty(\RR)}\) and \(\|h_0\|_{L^2(\RR)}\). \medskip

In order to show the wave breaking result, let us assume that the initial data is such that (cf. \cite{LeiRic})
\begin{equation}\label{condition1}
m(0)\leq -C,
\end{equation}
for some positive constant $C$ such that $C^{2}>4A$.
Then $2A-m(0)^{2}/2<0$ and
$$
m'(t)\bigg{|}_{t=0}\leq -m(0)^{2}-\frac{1}{2}m(0)^{2}+ 2A\leq -m(0)^{2}<0.
$$
Therefore there exists a sufficiently small $0<\delta$ such that $m'(t)<0$, for $0\leq t<\delta$.
This implies $m(t)<m(0)\leq -C$  and similarly $2A-m(t)^{2}/2<0$ for $0\leq t<\delta$.
Thus, if 
$0<t\leq \min\{\delta,A/B\}$ 
from \eqref{evol:min2:new} it holds that
$$
m'(t)\leq - m(t)^{2}-\frac{1}{2}m(t)^{2}+ 2A \leq - m(t)^{2}.
$$
This leads to
\begin{equation}\label{blowu2}
m(t)\leq \frac{m(0)}{m(0)t+1}.
\end{equation}
Taking the initial datum such that
\begin{equation}\label{condition2}
\frac{1}{-m(0)}\leq \frac{A}{B},
\end{equation}
note that \eqref{condition1}, \eqref{condition2} define a constant $H_{0}$ for which $m(0)$ satisfies \eqref{cond:ini}, and (\ref{blowu2}) implies the existence of a time $T_{b}$ where \eqref{waveb} holds.
\end{proof}

\section*{Acknowledgments}
\noindent  D.A-O is supported by the Spanish MINECO through Juan de la Cierva fellowship FJC2020-046032-I. 
A.D is supported by the Spanish Agencia Estatal de Investigaci\'{o}n under Research Grant PID2020-113554GB-I00/AEI/10.13039/501100011033 and by the Junta de Castilla y Le\'{o}n and FEDER funds (EU) under
Research Grant VA193P20. R.G-B is supported by the project "Mathematical Analysis of Fluids
and Applications" Grant PID2019-109348GA-I00 funded by MCIN/AEI/ 10.13039/501100011033 and
acronym "MAFyA". This publication is part of the project PID2019-109348GA-I00 funded by MCIN/ AEI
/10.13039/501100011033. This publication is also supported by a 2021 Leonardo Grant for Researchers
and Cultural Creators, BBVA Foundation. The BBVA Foundation accepts no responsibility for the opinions, statements, and contents included in the project and/or the results thereof, which are entirely the
responsibility of the authors.

\begin{footnotesize}

\end{footnotesize}
\vspace{2cm}

\end{document}